\documentclass[11pt, letterpaper,reqno]{amsart}
\usepackage[left=1in,right=1in,bottom=0.5in,top=0.8in]{geometry}
\usepackage{amsfonts}
\usepackage{amsmath, amssymb}
\usepackage{graphicx} 
\usepackage[font=small,labelfont=bf]{caption}
\usepackage{epstopdf}
\usepackage[colorlinks, linkcolor=black, urlcolor=blue]{hyperref}
\usepackage{xcolor}
\usepackage{amsthm}
\usepackage{pgfplots}
\usepackage{listings}
\usepackage{longtable}
\usepackage{mathrsfs}
\usepackage[utf8]{inputenc}  
\usepackage[T1]{fontenc}
\usepackage{indentfirst} 
\usepackage{enumerate} 
\usepackage{fancyhdr} 
\usepackage{pdfpages} 
\usepackage{verbatim} 
\usepackage{float} 
\usepackage[numbers]{natbib}

\newtheorem{teo}{Theorem}[section]
\newtheorem{prop}[teo]{Proposition}
\newtheorem{lema}[teo]{Lemma}
\newtheorem{cor}[teo]{Corollary}
\theoremstyle{definition} 
\newtheorem{defi}{Definition}[section]

\DeclareMathOperator{\G}{GL}

\DeclareMathOperator{\V}{SL}
\DeclareMathOperator{\W}{M}

\newcommand{\N}{_n(\mathbb{R})}
\newcommand{\NN}{_{n+1}(\mathbb{R})}
\newcommand{\E}{S^{n-1}}
\newcommand{\B}{B^n}

\newcommand{\Y}{_{n,+}(\mathbb{R})}
\DeclareMathOperator{\A}{vol}
\newcommand{\RR}{\times\mathbb{R}^n}

\DeclareMathOperator{\Id}{Id}
\DeclareMathOperator{\R}{Sym}
\newcommand{\s}{{}^{(s)}}
\DeclareMathOperator{\fa}{\bf{f1}}
\DeclareMathOperator{\fb}{\bf{f2}}
\DeclareMathOperator{\fc}{\bf{f3}}
\DeclareMathOperator{\fd}{\bf{f4}}
\DeclareMathOperator{\ga}{\bf{g1}}
\DeclareMathOperator{\gb}{\bf{g2}}
\DeclareMathOperator{\gc}{\bf{g3}}
\DeclareMathOperator{\gd}{\bf{g4}}
\DeclareMathOperator{\gee}{\bf{g5}}
\DeclareMathOperator{\interior}{int}
\author[]{Fernanda Moreira Baêta}
\address{\tiny{INSTITUT FÜR DISKRETE MATHEMATIK UND GEOMETRIE, TECHNISCHE UNIVERSITÄT WIEN, WIEDNER HAUPT-STRASSE 8-10/1046, 1040 WIEN, AUSTRIA}}
\email{fernanda.baeta@tuwien.ac.at}
\subjclass[2020]{}
\title[]{Constructive Decompositions of the Identity for Functional John Ellipsoids}
\pgfplotsset{compat=1.18} 
\begin{document}
\pretolerance10000
\begin{abstract}
We consider  functional ellipsoids in the sense defined by  Ivanov and  Naszódi \cite{14} and we study the problem of constructing a decomposition of the identity similar to the one given by Fritz John in his fundamental theorem.    
\end{abstract}
\maketitle

\section{Introduction and Main Results}
In 1948, Fritz John established that every convex body $K \subset \mathbb{R}^n$  contains a  largest volume ellipsoid. When this ellipsoid is the $n$-dimensional unit Euclidean ball $B^n$, the body $K$ is said to be in John position. John's theorem further guarantees the existence of a finite set of points $\{\xi_1,\dots, \xi_m\}\subset S^{n-1}\cap\partial K$, positive weights $\{c_1,\dots, c_m\}$, and a scalar $\lambda\neq 0$ such that the following conditions hold
\begin{align}\label{eq123}
\sum_{i=1}^m c_i \xi_i\otimes \xi_i=\lambda \Id \qquad  \mbox{and} \qquad \sum_{i=1}^m c_i \xi_i=0.
\end{align}
Here $v\otimes w$ denotes the rank-one matrix $vw^T$, $\Id$ is the $n\times n$ identity matrix, $S^{n-1}$ is the unit Euclidean sphere, and $\partial K$ is the  boundary of $K$ (see [\citealp{4}, Application 4, pag. 199 - 200]). 
This necessary condition also holds for the dual case where the unit Euclidean ball is the ellipsoid with minimum volume containing $K$. We say in this case that $K$ is in Löwner position. The decomposition  \eqref{eq123} is often referred to as the  decomposition of the identity.  As shown by Ball \cite{ball}, the existence of a measure $\mu_K$, supported on $S^{n-1}\cap \partial K$, satisfying
\begin{align}\label{dec}
\int_{\E}(\xi\otimes\xi)d\mu_K = \lambda \Id \qquad \mbox{and} \qquad \int_{\E}\xi d\mu_K = 0,
\end{align}
for some $\lambda\neq 0$, ensures that $K$ is in John position if $B^n \subseteq K$, or in Löwner position if $K\subseteq B^n$. A measure satisfying \eqref{dec} is called isotropic and centered, respectively.

The John and Löwner ellipsoids form a cornerstones of modern convex geometry, and many problems has been solved using properties of these  objects. The relationship between isotropic measures and extremal positions has been widely studied, including by  \cite{1009, 1010, 1011}. Extensions to related minimization problems appear in \cite{1012, 1013, 1014,  1016, 1015}. More recently, the theory of ellipsoids has been extended to the space of log-concave functions.

In 2018, Alonso-Gutiérrez, Gonzales Merino, Jiménez and Villa \cite{15} extended  the  notion of the John ellipsoid to  log-concave functions. A function $\varphi: \mathbb{R}^n\rightarrow (-\infty, +\infty]$ is  \textit{convex} if, for all $x,y\in\mathbb{R}^n$ and $\lambda\in [0,1]$,  
$$\varphi(\lambda x+(1-\lambda)y)\leq \lambda \varphi(x)+(1-\lambda)\varphi(y).$$ When $h=e^{-\varphi}$ for a convex function $\varphi$, $h$ is called a \textit{log-concave function}. Given an integrable log-concave function $h: \mathbb{R}^n \to [0, \infty)$, they defined its John ellipsoid as follows:  For a fixed constant $ \beta \in (0, \|h\|_\infty) $, consider the superlevel set $ \{x \in \mathbb{R}^n : h(x) \geq \beta\} $, which is a bounded convex set with non-empty interior. For each level $ \beta > 0 $, let $ \mathcal{E} $ be the maximal-volume ellipsoid contained within this superlevel set. They proved the existence of a unique height $ \beta_0 \in [0, \|h\|_\infty] $ that maximizes $ \beta_0 \A_n(\mathcal E) $, where $ \A_n $ denotes the Lebesgue measure. The John ellipsoid of $h$ is then defined as the function $ \mathcal{E}^{\beta_0}(x) = \beta_0 {1}_{\mathcal{E}}(x) $, where $ {1}_{\mathcal{E}} $ is the indicator function of $ \mathcal{E} $.

In 2019, Li, Schütt, and Werner \cite{el} introduced the dual notion of the Löwner ellipsoid for log-concave functions. They showed that for any non-degenerate, integrable log-concave function $ h $, there exists a unique pair $ (A_0, t_0) $, where $ A_0 $ is an invertible affine transformation and $ t_0 \in \mathbb{R} $, such that  
\[
\int_{\mathbb{R}^n} e^{-|A_0x|_2 + t_0} \, dx = \min \left\{ \int_{\mathbb{R}^n} e^{-|Ax|_2 + t} \, dx : e^{-|Ax|_2 + t} \geq h(x) \right\}.
\]
Here, $ |x|_2 $ denotes the Euclidean norm of $ x \in \mathbb{R}^n $. The function $ e^{-|A_0x|_2 + t_0} $ is called the  Löwner function of $h$.

We say that $h:\mathbb{R}^n\rightarrow \mathbb{R}$ is \textit{upper semicontinuous} if
$$ \limsup_{k\rightarrow +\infty} h(x_k)\leq h(x),$$
whenever $x_k\rightarrow x$ as $k\rightarrow +\infty$.  A log-concave function $h$ on $\mathbb{R}^n$ is said to be \textit{proper}  if $h$ is upper semicontinuous and has finite positive integral.  We will say that a function $f:\mathbb{R}^n\rightarrow \mathbb{R}$  is below a function $g:\mathbb{R}^n\rightarrow \mathbb{R}$  if $f(x)\leq g(x)$ for all $x\in\mathbb{R}^n$. 

Recently, in 2021,  Ivanov and  Naszódi \cite{14} also extended the notion of the John ellipsoid to the setting of logarithmically concave functions. Unlike the first ones, they defined a class of functions on $\mathbb{R}^n$  indexed by $s>0$. First they fix  a log-concave function $h:\mathbb{R}^n\rightarrow [0,\infty)$   and a parameter $s>0$. Later, they prove  that there exists (and is unique within the set of log-concave functions) a log-concave  function with the largest integral under the condition that it is pointwise less than or equal to $h^{1/s}$. This function is called the \textit{John $s$-function of $h$}. 
In [\citealp{14}, Theorem 6.1], it is shown that as $s\rightarrow 0$, the John $s$-functions  converge to characteristic functions of ellipsoids, thereby establishing a relationship between the first \cite{15} and second approach \cite{14}. Furthermore, the authors study the behavior of the John $s$-functions as $s \to \infty$, demonstrating that the limit may only be a Gaussian density, which is not necessarily unique.

Ivanov and Tsiutsiurupa \cite{functional_Lowner}, in 2021,  studied the  dual problem of the John $s$-function of a log-concave function defined in \cite{14} and introduced the Löwner $s$-function. They combined ideas  from \cite{el} and \cite{14} as following. For a function $\psi: [0,+\infty)\rightarrow (-\infty, +\infty]$, they considered the class of functions $\alpha e^{-\psi(|A(x-a)|_2)}$, where  $A$  is invertible, $\alpha>0$, and $a\in\mathbb{R}^n$,  as the class of ``affine'' positions of the function $x\mapsto e^{-\psi(|x|_2)}, x\in\mathbb{R}^n$.  Note that these problems are related because the classes of ``affine'' positions of the characteristic function of the unit ball were considered in \cite{15}, while  the classes of ``affine'' positions of the
 function $x\mapsto e^{-|x|_2}, x\in\mathbb{R}^n$, were studied in  \cite{el}. For a function $\psi: [0,+\infty)\rightarrow (-\infty, +\infty]$ such that $e^{-\psi(|t|)}, t\in \mathbb{R}$, is an upper semicontinuous log-concave function with a finite positive integral, and   an upper semicontinuous log-concave function $h:\mathbb{R}^n\rightarrow [0,+\infty)$ of finite positive integral, they studied the following optimization problem
\begin{align*}
\int_{\mathbb{R}^n} \alpha_0 e^{-\psi(|A_0(x-a_0)|_2)}dx= \min\left\{\int_{\mathbb{R}^n} \alpha e^{-\psi(|A(x-a)|_2)}dx:  h(x)\leq \alpha e^{-\psi(|A(x-a)|_2)}\right\}.
\end{align*}

The height function of the $(n+1)$-dimensional unit ball $B^{n+1}\subset \mathbb{R}^{n+1}$, given by $\hslash_{B^{n+1}}(x)=\sqrt{1-|x|_2^2}$ if $x\in B^n$ and 0 otherwise,  is a proper log-concave function.
The main advantage of the framework presented in \cite{14}  is that it implies  a ``decomposition of the identity'' as in \eqref{eq123}. Namely,

\begin{teo}[\cite{14}, Theorem 5.2]\label{decom_iden_h}
Let $h$ be a proper log-concave function on $\mathbb{R}^n$ and $s>0$. Assume $\hslash^s_{B^{n+1}}\leq h$. Then  the following conditions are equivalent:
\begin{enumerate}
\item[$(1)$]  The function $\hslash^s_{B^{n+1}}$ is the John $s$-function of $h$;
\item[$(2)$] There exist points $u_1,\dots, u_k \in B^n\subset \mathbb{R}^d$ and positive weights $c_1,\dots, c_k$, such that 
\begin{enumerate}
\item[$(a)$] $h(u_i)=\hslash^s_{B^{n+1}}(u_i)$ for all $i=1,\dots, k$;
\item[$(b)$] $\sum_{i=1}^k c_iu_i\otimes u_i=\Id$;
\item[$(c)$] $\sum_{i=1}^k c_i h(u_i)^{1/s} h(u_i)^{1/s}=s$;
\item[$(d)$] $\sum_{i=1}^k c_i u_i=0.$
\end{enumerate}
\end{enumerate}  
\end{teo}

Similar to the decomposition of the identity in the geometric case,  the authors of \cite{14} guarantee the existence of a measure satisfying $(a)-(d)$ in the theorem above, although the proof is not constructive.  In \cite{7}, the authors presented a constructive proof of John’s theorem in the geometric setting, using a simple finite-dimensional minimization problem.

In this context, assuming that $\hslash^s_{B^{n+1}}$ is the John $s$-function of $h$, the purpose of this paper is to present a constructive proof of the necessity part of Theorem \ref{decom_iden_h}, using  as in \cite{7} a simple finite dimensional minimization problem.

Throughout this paper, we denote by $\W_d(\mathbb{R})$  the vector space of $d \times d$  matrices, and by $\G_d(\mathbb{R}) $ and $ \V_d(\mathbb{R})\subseteq \W_d(\mathbb{R})$ the subsets of invertible matrices and matrices
with determinant 1, respectively.  The set of symmetric matrices in $\W_d(\mathbb{R})$ will be denoted by $\R_d(\mathbb{R})$, and Sym$_{d,+}(\mathbb{R})$ will denote  the subgroup of symmetric and positive-definite matrices. The  $(d-1)$-dimensional Hausdorff measure is denoted by  $\mathcal{H}^{d-1}$,
and int $A$ represents the interior of $A$.

The hyperplane in $\mathbb{R}^{n+1}$ spanned by the first $n$ standard basis vectors is identified with $\mathbb{R}^n$. We say that a set $\bar{C}\subset \mathbb{R}^{n+1}$  is $n$-symmetric if $(x,t)\in \bar{C}$ implies $(x,-t)\in \bar{C}$. Throughout this paper,  $\det$  denotes the determinant function defined on $\W\N$, while    the determinant function defined on  $\W\NN$   will be denoted by $\mbox{det}_{n+1}$.  The trace function in either matrix space $\W\N$ or $\W\NN$ will be denoted simple by tr.

Following \cite{14}, for a square matrix  $A\in\W\N$ and a scalar $\alpha\in \mathbb{R}$,  $A\oplus \alpha$ denotes the $(n+1)\times (n+1)$ matrix
\begin{displaymath}
A\oplus \alpha = \begin{pmatrix}
A&0\\
0&\alpha
\end{pmatrix}. 
\end{displaymath} 
Note that  $\mbox{det}_{n+1}(A\oplus\alpha)=\alpha\det(A)$ and $\mbox{tr}(A\oplus\alpha)=\alpha+\mbox{tr}(A)$. We introduce the $\frac{(n+1)(n+2)}{2}+n$ dimensional vector space
\begin{align*}
\mathcal{M}=\{(\bar{A},a): \bar{A}\in\R\NN, a\in\mathbb{R}^n\},
\end{align*}
the subspace
\begin{align*}
\mathcal{E}= \{(A\oplus\alpha, a)\in\mathcal{M}: A\in \R\N, \alpha>0\},
\end{align*}
and the convex cone
\begin{align*}
\mathcal{E}_+= \{(A\oplus\alpha, a)\in\mathcal{E}: A \ \mbox{is defined positive }, \alpha>0\}.
\end{align*}

While  Alonso-Gutiérrez, Gonzales Merino, Jiménez, and Villa in \cite{15}  define an ellipsoid as $A(\B)+a$,
where $A\in \W\N$ is a positive-definite matrix  and $a\in\mathbb{R}^n$, in \cite{14} they consider $n$-symmetric ellipsoids in $\mathbb{R}^{n+1}$. Since every $n$-symmetric ellipsoid in $\mathbb{R}^{n+1}$ is uniquely represented as
$$(A\oplus \alpha)B^{n+1}+a,$$
where $A\in\G\N$, thus by Polar Decomposition, $\mathcal{E}_+$ uniquely determines each $n$-symmetric ellipsoid of $\mathbb{R}^{n+1}$. Here $\bar{v}+a$, where $\bar{v}\in\mathbb{R}^{n+1}$ and $a\in\mathbb{R}^n$, denotes $\bar{v}+(a,0)$. To improve readability,  we use a bar to denote subsets of $\mathbb{R}^{n+1}$. 

Consider the  $(n+1) \times (n+1)$ matrix $M\oplus \beta$,
where $ M\in\W\N$  and  $\beta \in(0,+\infty)$. We define the \textit{$s$-determinant}  of $M\oplus \beta$ by
\begin{align*}
 ^{(s)}\mbox{det}_{n+1}(M\oplus \beta)= \beta ^s\det(M),
\end{align*}
and the \textit{$s$-trace} of $M\oplus \beta$ by
\begin{align*}
 ^{(s)}\mbox{tr}(M\oplus \beta)=s\beta + \mbox{tr}(M).
\end{align*}

Based on these definitions, we define the following sets 
$$ ^{(s)}\V\NN=\left\{M\oplus \beta\in \W\NN : {}^{(s)}\mbox{det}_{n+1}(M\oplus \beta) = 1\right\},$$
$$\s\R_{n+1,0}(\mathbb{R})=\{M\oplus \beta\in \R\NN: \s\mbox{tr}(M\oplus \beta)=0\},$$
and
$$^{(s)} \mathcal{E}_+= \{(A\oplus\alpha, a)\in\mathcal{M}: A\in \R\Y, \alpha>0, a\in\mathbb{R}^n  \mbox{ and } {}^{(s)}\mbox{det}_{n+1}(A\oplus \alpha)\geq 1 \}. $$
These sets are related to the  John $s$-function of the log-concave function $h$ defined in \cite{14}. In  Section \ref{preliminary}, these definitions will be clarified.

Since the goal of this is to construct a measure satisfying the items of condition (2) of Theorem \ref{decom_iden_h}, we  introduce the following definition.
\begin{defi}\label{s-isotropic}
A measure $\mu$ on the unit Euclidean ball $\B$ is said to be \textit{$s$-isotropic} if for some $\lambda\neq 0$, it holds that 
$$\int_{\B} (u\otimes u \oplus (1-|u|_2^2))d\mu= \lambda (\Id\oplus s),$$
and it is called \textit{centered} if
$$\int_{\B}ud\mu=0.$$
\end{defi}

Let $g: V\rightarrow \mathbb{R}$ be a function defined on a vector subspace $V\subseteq \mathbb{R}^d$.  We say that $g$ is \textit{coercive} if
$$\lim_{|x|_2\rightarrow \infty} g(x)=+\infty.$$
Note that every continuous and coercive function admits a global minimum. Denote the standard inner product by $\langle u, v\rangle =\sum_{i=1}^n u_iv_i$, and consider the set 
$$\mathcal{W}= \{F:\mathbb{R}\rightarrow [0,\infty):  \ F \mbox{ is  non-decreasing, convex, strictly convex in } [0,\infty), \mbox{ and } F'(0)>0\}.$$

Our main result is the following.
\begin{teo}\label{teo3.3}
Let $h:\mathbb{R}^n\rightarrow\mathbb{R}$ be a proper log-concave function and $\hslash_{B^{n+1}}$ its John $s$-function. Let  $\nu$ be any finite positive, non-zero measure in $\B$ with support inside  the subset $\Lambda=\{x\in \B: h(x)^{1/s}=\hslash_{B^{n+1}}(x)\}$, and let   $F\in \mathcal{F}$ be any $C^1$ function. Consider the convex functional $\bar{I}_\nu: \mathcal{E}\rightarrow \mathbb{R}$ defined by 
$$\bar{I}_\nu(M\oplus\beta,w)= \int_{\B}h(x)^{1/s}F\left(\dfrac{\langle x,Mx+w\rangle}{h(x)^{2/s}}+\beta\right)d\nu(x).$$ 
If the restriction of $\bar{I}_\nu$ to $\s\R_{n+1,0}(\mathbb{R})\RR$ is coercive, then for any global minimum $(M_0\oplus\beta_0,w_0)$, the measure 
$$\dfrac{1}{h(x)^{1/s}}F'\left(\dfrac{\langle x,M_0x+w_0\rangle}{h(x)^{2/s}}+\beta_0\right)d\nu(x)$$
is non-negative, non-zero, centered, and $s$-isotropic.
\end{teo}
Note that $\Lambda$ is the set of points where the  function $h$ coincides with its John $s$-function.  
Assume that $\Lambda$ is finite and that  $\nu$ is the counting measure $c$. As a consequence of the previous theorem, we obtain the following result.

\begin{cor}
Let $h:\mathbb{R}^n\rightarrow\mathbb{R}$ be a proper log-concave function and $\hslash_{B^{n+1}}$ its John $s$-function. Assume
$$\Lambda=\{x\in \B: h(x)^{1/s}=\hslash_{B^{n+1}}(x)\} = \{x_1,\dots, x_m\}.$$
Let $F \in \mathcal{F}$ be any $C^1$ function. Consider the convex functional  $\bar{I}_c: \mathcal{E}\rightarrow \mathbb{R}$ defined by 
$$\bar{I}_c(M\oplus\beta,w) = \sum_{i=1}^m h(x_i)^{1/s}F\left(\dfrac{\langle x_i,Mx_i+w\rangle}{h(x_i)^{2/s}}+\beta\right).$$
If the restriction of $\bar{I}_c$ to $\s\R_{n+1,0}(\mathbb{R})\RR$ is coercive, then for any global minimum $(M_0\oplus\beta_0,w_0)$, the numbers 
$$c_i=\dfrac{1}{h(x_i)^{1/s}}F'\left(\dfrac{\langle x_i,M_0x_i+w_0\rangle}{h(x_i)^{2/s}}+\beta_0\right), i=1,\dots, m,$$
together with the vectors $x_i, i=1,\dots,m$,  satisfy  the conditions of item $(2)$ of Theorem \ref{decom_iden_h}.
\end{cor}

By Lemma \ref{unico_minimo} in Section 2, if $\bar{I}_\nu$ has an isolated local minimum,  it must be coercive. This implies that if a minimum is found, local coercivity can be established. Furthermore, the next theorem shows that the coercivity condition  in Theorem \ref{teo3.3} corresponds to a generic situation and  depends only   on the choice of  measure $\nu$. Using a similar idea to  items 1 and 2 of Theorem 1.6 in \cite{7}, we obtain the following result.  
\begin{teo}\label{teo6.3}
Let $h:\mathbb{R}^n\rightarrow\mathbb{R}$ be a proper log-concave function and $\hslash_{B^{n+1}}$ its John $s$-function. Consider $F$ as in Theorem \ref{teo3.3}. The following statements are equivalent
\begin{enumerate}
\item[(a)] The restriction of $\bar{I}_\nu$ to $\s\R_{n+1,0}(\mathbb{R})\RR$ is coercive;
\item[(b)] For every $(M\oplus\beta, w)\in\left(\s\R_{n+1,0}(\mathbb{R})\RR\right)\setminus \{(0,0)\}$, we have  
$$\nu\left(\left\{x\in\B: h(x)^{1/s}=\hslash_{B^{n+1}}(x) \ \mbox{and} \ \left(\dfrac{\langle x, Mx+w\rangle}{h(x)^{2/s}}+\beta\right)> 0\right\}\right)>0.$$
\end{enumerate}
\end{teo}

As mentioned earlier, we aim to construct a centered and $s$-isotropic measure from a log-concave function $h$ on $\mathbb{R}^n$ whose   John $s$-function is $\hslash_{B^{n+1}}$. We propose a simple finite-dimensional minimization
problem as following.  For any $r\in(1/2,1) $  and any function $\gamma:\mathbb{R}\rightarrow \mathbb{R}$, consider the family of functions $\gamma_r(s)= \gamma\left(\dfrac{s-1}{1-r}\right)$ and  two measurable functions $f,g:\mathbb{R}\rightarrow \mathbb{R}$.

We define the functional $\bar{L}_r:  \W\N\oplus (0,+\infty)\RR\rightarrow\mathbb{R}$ by
\begin{align*}
\bar{L}_r(A\oplus\alpha,v) = \dfrac{1}{1-r}\int_{\mathbb{R}^n}\int_0^\infty f_r\left(\dfrac{\alpha y}{h(Ax+v)^{1/s}}\right)g_r\left( \dfrac{|x|_2^2+y^2-1}{2h(x)^{2/s}}+1\right)dydx,
\end{align*}
where $\W\N\oplus (0,+\infty)$ denotes the set of matrices $A\oplus \alpha \in \W\NN$. For $\tilde{A}_r= \Id+(1-r)M$ and $\tilde{\alpha}_r= 1+ (1-r)\beta$, we define the functional $\bar{I}_r: \bar{B}_r\RR\subseteq \mathcal{E}\rightarrow \mathbb{R}$ by
\begin{align*}
\nonumber \bar{I}_r(M\oplus \beta, w)& =\dfrac{\tilde{\alpha}_r^{s-1}}{1-r}\int_{\mathbb{R}^n}\int_0^\infty  f_r\left(\dfrac{y}{h(x)^{1/s}}\right) g_r\left(\dfrac{|\tilde{A}_r^{-1}(x-(1-r)w)|_2^2+(\tilde{\alpha}_r^{-1}y)^2-1}{2h(\tilde{A}_r^{-1}(x-(1-r)w))^{2/s}}+1\right)dydx,
\end{align*}
where  $\bar{B}_r=\{M\oplus\beta\in \mathcal{E}: M\in\R\N \ \mbox{is such that $(\Id+(1-r)M)$ is invertible}\}$.

The functionals are related for $(A\oplus \alpha, w)\in \s \V\NN\RR$ as follows
\begin{align*}
 \bar{I}_r\left(\dfrac{A\oplus \alpha-\bar{\Id}}{1-r}, \dfrac{w}{1-r}\right)&= \bar{L}_r(A\oplus \alpha,w). 
\end{align*}

The idea is to minimize the functional $\bar{L}_r$ over all  $n$-
symmetric positions of the unit Euclidean ball $B^{n+1}$ and thus obtain a sequence of measures that weakly converge to a  centered and $s$-isotropic measure. Consider the following lemma. 

\begin{lema}\label{lema_auxialiar}
Let $(A_r\oplus \alpha_r,v_r)$ be  a global minimum of the restriction of $\bar{L}_r$ to $\mathcal{E}_+\cap (\s\V\NN\RR)$. Then,   there exists $\lambda_r\neq 0$ such that
\begin{align}
\nonumber (1-r)\lambda_r(\Id\oplus s) = \dfrac{\alpha_r^{s-1}}{1-r}\int_{\mathbb{R}^n}\int_0^\infty &(f')_r\left(\dfrac{ y}{h(x)^{1/s}}\right)g_r\left(\dfrac{|A_r^{-1}(x-v_r)|_2^2+(\alpha_r^{-1}y)^2-1}{2h(A_r^{-1}(x-v_r))^{2/s}}+1\right)\dfrac{y}{h(x)^{3/s}}\\ \nonumber 
& \times \left(- \nabla h(x)^{1/s}h(x)^{1/s}\otimes x\oplus h(x)^{1/s}h(x)^{1/s}\right)dydx,
\end{align}
and
\begin{align}
\nonumber 0  =  \dfrac{\alpha_r^{s-1}}{1-r}\int_{\mathbb{R}^n}\int_0^\infty &(f')_r\left(\dfrac{ y}{h(x)^{1/s}}\right)g_r\left(\dfrac{|A_r^{-1}(x-v_r)|_2^2+(\alpha_r^{-1}y)^2-1}{2h(A_r^{-1}(x-v_r))^{2/s}}+1\right)\dfrac{y}{h(x)^{3/s}} \\ \nonumber
&  \times\left(- \nabla h(x)^{1/s}h(x)^{1/s}\right)dydx. 
\end{align}
\end{lema}
Note that if $x\in \interior \B$, then
$$\nabla \hslash_{B^{n+1}}(x)=-\dfrac{x}{\hslash_{B^{n+1}}(x)},$$
and  for every $x\in\interior \B$  such that $h(x) = \hslash_{B^{n+1}}^s(x)$, we have that $-\nabla h(x)^{1/s}h(x)^{1/s} = x$.
Now consider the set $\Lambda=\{x\in\B: h(x) = \hslash_{B^{n+1}}^s(x)\}$, a Borel set $B\subseteq{R}^n$, and the measure 
\begin{align}\mu_r(B)=\int_B\dfrac{\alpha_r^{s-1}}{1-r}\int_0^\infty (f')_r\left(\dfrac{ y}{h(x)^{1/s}}\right)g_r\left(\dfrac{|A_r^{-1}(x-v_r)|_2^2+(\alpha_r^{-1}y)^2-1}{2h(A_r^{-1}(x-v_r))^{2/s}}+1\right)\dfrac{y}{h(x)^{3/s}}dydx.\label{eq5.30}
\end{align}

It holds that
\begin{align}
\int_{\Lambda} \left(- \nabla h(x)^{1/s}h(x)^{1/s}\otimes x\oplus h(x)^{1/s}h(x)^{1/s}\right)d\mu_r(x)= \int_{\Lambda}\left(x\otimes x\oplus h(x)^{1/s}h(x)^{1/s}\right)d\mu_r(x),\label{eqI}
\end{align}
and 
\begin{align}
 \int_{\Lambda} \left(- \nabla h(x)^{1/s}h(x)^{1/s}\right)d\mu_r(x)=\int_{\Lambda}  xd\mu_r(x).\label{eqII}
\end{align}

We will show that the measure $\mu_r(\cdot)$ concentrates near $\Lambda$ as $r\rightarrow 1^-$ and weakly converges to a centered and $s$-isotropic measure. In order to construct this measure, we will assume the following properties for $f$ and $g$:

\begin{enumerate}
   \item[\textbf{f1}] $f$ is locally Lipschitz;
    \item[\textbf{f2}] $f$ is convex;
    \item[\textbf{f3}] $f(x)=0$ for $x\leq-1$;
    \item[\textbf{f4}] $f$ is strictly increasing in $[-1, \infty)$.
   \item[\textbf{g1}] $g$  is locally Lipschitz;
    \item[\textbf{g2}] $g$ is non-increasing;
    \item[\textbf{g3}] $g(x)=1$ for $x\leq-1$;
    \item[\textbf{g4}] $g(x)>0$ for $x\in(-1,1)$;
    \item[\textbf{g5}] $g(x)=0$ for $x\geq 1$.
\end{enumerate}

Two simple functions satisfying properties $\fa$ to $\gee$ are
\begin{equation*}
   f(x) = \begin{cases}
      0, & \text{if } x \leq -1 \\
      x+1, & \text{if } x > -1
   \end{cases}, 
   \quad
   g(x) = \begin{cases}
      1, & \text{if } x \leq -1 \\
      \frac{1 - x}{2}, & \text{if } x \in (-1, 1) \\
      0, & \text{if } x \geq 1
   \end{cases}.
\end{equation*}

This choice of functions guarantees the following results.
\begin{teo}\label{teo7.3}
Let $h:\mathbb{R}^n\rightarrow \mathbb{R}$ be a proper log-concave function and $\hslash_{B^{n+1}}$ its John $s$-function. Consider functions $f$ and $g$ that satisfy all the properties $\fa$ to $\gee$. Then for every $r\in (1/2,1)$, the restriction of $\bar{L}_r$ to $\mathcal{E}_+\cap (\s\V\NN\RR)$ has a unique minimum $(A_r\oplus\alpha_r,v_r)$, up to horizontal translation, with $\lim_{r\rightarrow 1^-}(A_r\oplus\alpha_r,v_r)=(\bar{\Id},0)$. Likewise, the restriction of $\bar{I}_r$ to
$$\dfrac{\mathcal{E}_+\cap (\s\V\NN\RR)-\bar{\Id}\RR}{1-r} $$ 
has the unique minimum $(M_r\oplus\beta_r,w_r)=\left(\dfrac{A_r\oplus\alpha_r-\bar{\Id}}{1-r},\dfrac{v_r}{1-r}\right)$, up to horizontal translation, with $\s\mbox{tr}\left(\dfrac{M_r\oplus\beta_r}{||M_r\oplus\beta_r||_F}\right)\rightarrow 0$ as $r\rightarrow 1^-$.
\end{teo}

\begin{teo}\label{teo8.3}
Assume that all the properties $\fa$ to $\gee$ are satisfied. The functional $\bar{I}_r(M\oplus\beta,w)$ is extended continuously to $r=1$ as 
$$\bar{I}_1(M\oplus\beta, w) = \int_{\Lambda} h(x)^{1/s} F\left( \dfrac{\langle x,Mx+w\rangle}{h(x)^{2/s}}+\beta\right)dx,$$
where $\Lambda=\{x\in \B: h(x)^{1/s}=\hslash_{B^{n+1}}(x)\}$ and $F$ is the convolution $F(x)=f\ast\bar{g}(x),\bar{g}(x)=g(-x)$, satisfying the conditions of Theorem \ref{teo3.3}. Moreover, $\bar{I}_r\rightarrow \bar{I}_1$ as $r\rightarrow 1^-$, uniformly in compact sets.
\end{teo}

To calculate the limit of the measure \eqref{eq5.30}, one needs to compute $(A_r\oplus \alpha_r, v_r)$ for  $r$ close to 1. The content of the last theorem guarantees that the necessary information for computing the $s$-isotropic measure is contained in $(M_0\oplus \beta_0,w_0)$, and hence  Theorem \ref{teo3.3} follows directly.

\begin{teo}\label{teo9.3}
Assume all the properties $\fa$ to $\gee$ are satisfied, and the function $\bar{I}_1$ restricted to $\s\R_{n+1,0}(\mathbb{R})\RR$ has a unique global minimum $(M_0\oplus\beta_0,w_0)$. Then 
$$ \left. \frac{\partial (A_r \oplus \alpha_r, v_r)}{\partial r} \right|_{r=1} = -(M_0 \oplus \beta_0, w_0). $$

In this case, if  $(\tilde{A}_r\oplus\tilde{\alpha}_r, \tilde{v}_r)$ is any curve in $\mathcal{E}_+$ of the form 
$$(\tilde{A}_r\oplus\tilde{\alpha}_r, \tilde{v}_r)= (\bar{\Id},0)+(1-r)(M_0\oplus\beta_0,w_0)+o(1-r),$$
the measure 
$$\dfrac{\tilde{\alpha}_r^{s-1}}{1-r}\int_0^\infty (f')_r\left(\dfrac{y}{h(x)^{1/s}}\right)g_r\left(\dfrac{|\tilde{A}_r^{-1}(x-v_r)|_2^2+(\tilde{\alpha}_r^{-1}y)^2-1}{h(\tilde{A}_r^{-1}(x-\tilde{v}_r))^{2/s}}+1\right)\dfrac{y}{h(x)^{3/s}}dydx$$
converges weakly to the centered and  $s$-isotropic 
 measure $$\dfrac{1}{h(x)^{1/s}}F'\left(\dfrac{\langle x, M_0x+w_0\rangle}{h(x)^{1/s}}+\beta_0\right)dx.$$
In particular, this is true for $(\tilde{A}_r\oplus\tilde{\alpha}_r,\tilde{v}_r)= (A_r\oplus\alpha_r,v_r)$, and for its linear part $(\tilde{A}_r\oplus\tilde{\alpha}_r,\tilde{v}_r)=(\bar{\Id}+(1-r)(M_0\oplus\beta_0), (1-r)w_0)$.
\end{teo}

The paper is organized as follows: In Section \ref{preliminary}, we introduce the theory of the John $s$-function as defined by Ivanov and Naszódi in \cite{14} and recall some basic results. In Section 3, we prove some technical properties of the functionals $\bar{L}_r$ and $\bar{I}_r$. Finally, in Section 4, we prove the main theorems.

\section{Notation and Preliminary Results}\label{preliminary}
See \cite{14} for more details on functional ellipsoids, and \cite{Vitor,72} for basic facts on convexity. Let $s>0$. For every $x \in \mathbb{R}^n$, we denote by $l_x$ the line in $\mathbb{R}^{n+1}$ perpendicular to $\mathbb{R}^n$ at $x$, and by $l$ the one-dimensional Lebesgue measure on $l_x$. 
The $s$-\textit{volume} of an $n$-symmetric Borel set $\bar{C}$ is defined by
\begin{align*}
^{(s)}\mu(\bar{C})= \int_{\mathbb{R}^n}\left[\dfrac{1}{2}l(\bar{C}\cap l_x)\right]^s dx,
\end{align*} 
and the $s$-\textit{marginal} of a Borel set $B\subset\mathbb{R}^{n}$ is defined by
\begin{align*}
^{(s)}\mbox{marginal}(\bar{C})(B)= \int_B\left[\dfrac{1}{2}l(\bar{C}\cap l_x)\right]^sdx,
\end{align*}
as defined in \cite{14}. Note that this marginal is a measure on $\mathbb{R}^n$. A straightforward computation shows that for any matrix $\bar{A}=A\oplus \alpha$ and  any $n$-symmetric set $\bar{C}$ in $\mathbb{R}^{n+1}$, the following holds
\begin{align}\label{eq2.2.3}
 ^{(s)}\mu(\bar{A}\ \bar{C})= 
 |\det(A)| |\alpha|^s  \ ^{(s)}\mu(\bar{C}).
\end{align}
Using \eqref{eq2.2.3}, the $s$-volume of an $n$-symmetric ellipsoid can be expressed as
\begin{align*}
^{(s)}\mu((A\oplus \alpha)B^{n+1}+a) = \ ^{(s)}\mu(B^{n+1})\alpha^s \mbox{det}(A),
\end{align*}
for any $(A\oplus \alpha,a)\in\mathcal{E}$.
Now, let $h:\mathbb{R}^n\rightarrow [0,+\infty)$  be a function and  let $s>0$. In \cite{14}, the  \textit{$s$-lifting} of $h$ is defined  as the $n$-symmetric set in
$\mathbb{R}^{n+1}$ given by
\begin{align*}
^{(s)}\bar{h}=\{(x,\xi)\in \mathbb{R}^{n+1}: |\xi|\leq h(x)^{1/s}\},
\end{align*}
and this set is such that $^{(s)}\mbox{marginal}\left( ^{(s)}\bar{h}\right)$ is the measure on $\mathbb{R}^n$ with density $h$. According to [\citealp{14}, Theorem 4.1], for $s>0$  and a proper log-concave function $h$ on $\mathbb{R}^n$, there exists a unique $n$-symmetric ellipsoid contained in the $s$-lifting of $h$  that has the maximum $s$-volume. This ellipsoid in $\mathbb{R}^{n+1}$ is called the \textit{John $s$-ellipsoid} of $h$ and  is denoted  by $\bar{E}(h,s)$. The $s$-marginal of $\bar{E}(h,s)$ is called \textit{John $s$-function} of $h$.

Let $(A\oplus \alpha, a)\in\mathcal{E}_+$.  We define the \textit{height} of the ellipsoid $\bar{E} = (A \oplus \alpha)B^{n+1} + a$ as $\alpha$, and the \textit{height function} of $\bar{E}$ as
\begin{align*}
\hslash_{\bar{E}}(x) = 
\begin{cases}
\alpha \sqrt{1 - \langle A^{-1}(x-a), A^{-1}(x-a) \rangle}, & \text{if } x \in A\B + a \\
0, & \text{otherwise}
\end{cases}.
\end{align*}
Note that the height function of an ellipsoid is a proper log-concave function and $\bar{E}\subset \s\bar{h}$ holds if and only if $\hslash_{\bar{E}} (x+a)\leq h(x+a)^{1/s}$ for all $x\in A\B$. Note that $\hslash_{\bar{E}(h,s)}$ is the density of  $s$-marginal of $\bar{E}(h,s)$.

The set $\s\V\NN$, defined in the previous section, is relevant because it allows us to consider  ellipsoids contained in the  \textit{$s$-lifting} of $h$ that have the same $s$-volume as the $(n+1)$-dimensional unit Euclidean ball $B^{n+1}$, since we are assuming that  $\hslash_{B^{n+1}}^s$ is the  John $s$-function of the log-concave function $h$. Additionally,  the set $\s\R_{n+1,0}(\mathbb{R})$ is significant because it forms the orthogonal complement of $(\Id\oplus s, 0)$ in $\mathcal{E}$ and this fact is used in the proof of Theorem \ref{teo3.3}. In Theorem \ref{teo7.3}, the convergence $(A_r \oplus \alpha_r, v_r) \to (\bar{\Id}, 0)$ implies that the position of the $s$-lifting of $h$ that minimizes $\bar{L}_r$ converges to the John $s$-position of $h$ as $r \to 1^-$.

Two inequalities that will be auxiliary in this paper are that for $A,B\in\G\N$  symmetric and positive-definite linear matrices and  for $\lambda\in(0,1)$, it holds
\begin{align}\label{det}
\det((1-\lambda)A+\lambda B)\geq \det(A)^{1-\lambda}\det(B)^{\lambda},
\end{align}
with equality  if and only if $A=B$. And for  $a,b>0$, we have 
\begin{align}\label{media}
\lambda a +(1-\lambda)b \geq a^{\lambda}b^{1-\lambda},
\end{align}
with equality if and only if $a = b$. This inequality is the well-known arithmetic-geometric inequality, or AM-GM for short.

Using these inequalities, it follows that $^{(s)}\mathcal{E}_+$ is a convex set. Indeed, for $A\oplus\alpha, B\oplus \beta \in \ ^{(s)} \mathcal{E}_+$ and $\lambda\in [0,1]$, we have 
\begin{align}
\nonumber \s \mbox{det}_{n+1}(\lambda(A\oplus\alpha)+(1-\lambda)(B\oplus \beta)) 
& = (\lambda \alpha+ (1-\lambda)\beta)^s \det(\lambda A + (1-\lambda)B)\\ \label{eqconvex*}
& \geq (\alpha^\lambda\beta ^{1-\lambda})^s\det(A)^\lambda\det(B)^{1-\lambda}\\ \nonumber 
& = (\alpha^s\det(A))^\lambda (\beta^s\det(B))^{1-\lambda}\\ \label{eqconve*1}
& \geq 1.
\end{align}

A direct computation shows that a function $h:\mathbb{R}^n\rightarrow \mathbb{R}$ is log-concave if and only if 
\begin{align}\label{log}
   h(\lambda x+(1-\lambda)y)\geq  h(x)^\lambda h(y)^{(1-\lambda)} 
\end{align}
for any $x,y\in \mathbb{R}^n$ and every $\lambda\in (0,1)$. For $\bar{A}\in \W_{n+1}(\mathbb{R})$,  the operator norm and the Frobenius norm are defined as  
$$||\bar{A}||_{op}= \sup_{|v|_2=1} |\bar{A} v|_2, \ \ ||\bar{A}||_F=\sqrt{\mbox{tr}(\bar{A}^T A)},$$
respectively. We equip $\mathcal{M}$ and its subspaces with an inner product defined as
$$\langle (\bar{A},v),(\bar{B},w)\rangle= \langle \bar{A}, \bar{B}\rangle_F + \langle v,w\rangle = \sum_{i,j}A_{i,j}B_{i,j}+ \sum_i v_iw_i,$$
and for simplicity, we write
\begin{align}\label{produto}
\langle(\bar{A},v), (\bar{B},w)\rangle = \langle \bar{A}, \bar{B}\rangle + \langle v,w\rangle.
\end{align}
For $(\bar{A}, v)\in \W_{n+1}(\mathbb{R})\RR$, we use $||(\bar{A}, v)||= \sqrt{||\bar{A}||_F^2+|v|_2^2}$ which is
the norm induced by the inner product \eqref{produto}.

Since $h$ is assumed to be a proper log-concave function, then there exists a convex function $\psi$ such that $h=e^{-\psi}$, satisfying the following properties:
\begin{itemize}
\item $\lim_{|x|_2\rightarrow \infty}\psi(x)=+\infty$ (otherwise, the integral of $e^{-\psi(x)}$ diverges to $+\infty$);
\item The set $\{x\in \mathbb{R}^n:  \psi(x)<+\infty\}$ has positive measure (otherwise, the integral of $e^{-\psi(x)}$  is  zero). 
\end{itemize}

A well-known result that will be useful is the  Taylor expansion (see, for example, [\citealp{taylor}, Theorem 5.21]). This result says that  a function   $f:\mathbb{R}^n\rightarrow \mathbb{R}$ of class $C^1$ around $x_0$ admits at $x_0$ the following Taylor expansion of order one 
\begin{align*}
f(x)=f(x_0)+\nabla f(x_0)(x-x_0)+o(|x-x_0|_2),
\end{align*}
where $x\rightarrow x_0, |x-x_0|_2$ denotes the Euclidean norm of $x-x_0, \lim_{x\rightarrow x_0}\frac{o(|x-x_0|_2)}{|x-x_0|_2}=0$ and $\nabla f(x_0)$ is the gradient of $f$ at $x_0$. 
Another useful result is that for  $u,v\in\mathbb{R}^d$ and $T\in \W_d(\mathbb{R})$, it holds
\begin{align}\label{eq}
\langle Tu, v\rangle= \langle T, v\otimes u\rangle. 
\end{align}

\begin{lema}[\cite{7}, Lemma 2.3]
A convex function $\psi:\mathbb{R}^d\rightarrow \mathbb{R}$ with an isolated local minimum must be coercive.\label{unico_minimo}
\end{lema}

\section{ Basic results}\label{basicresults}
Throughout this section, we  fix a proper log-concave function $h:\mathbb{R}^n\rightarrow \mathbb{R}$ such that $\hslash_{B^{n+1}}$ is its John $s$-function. Due to the good properties of the functions $f$ and $g$, we will  have good properties for the functional $\bar{L}_r$ and $\bar{I}_r$, as well as the convex* property obtained in  Proposition \ref{prop14.3.*}, which allows us to show that these functionals have a unique minimum, up to  horizontal translation, restricted to certain sets. The following result is a straightforward consequence of  Rademacher’s Theorem and the
Dominated Convergence Theorem.
\begin{prop}\label{Prop12.3}
Assume $\fa, \ga, \gee$ are satisfied, then $\bar{L}_r,\bar{I}_r$, and $\bar{I}_1$ are $C^1$ for $r\in(1/2,1)$.
\end{prop}

\begin{prop}\label{prop13.3}
Assume $\fb,\fc,\fd,\gc$, then the family of functionals $\bar{L}_r$ restricted to $\s\mathcal{E}_+\RR$ is coercive, uniformly for $r\in(1/2,1)$. 
\end{prop}

\begin{proof}
Let $(x,y)\in B^{n+1}$ with $ y\geq 0$. Then
\begin{align*}
   \dfrac{|x|_2^2+y^2-1}{2h(x)^{2/s}}+1\leq 1 
\end{align*}
and by \textbf{g2} it holds that 
\begin{align*}
    g_r\left(\dfrac{|x|_2^2+y^2-1}{2h(x)^{2/s}}+1\right)\geq g_r(1)=g(0).
\end{align*}

Using  $\gd$, that $f, g$ are non-negative, and that $r > 1/2$, we obtain 
\begin{align*}
\bar{L}_r(A\oplus \alpha, v)& \geq \frac{1}{1-r} \int_{B^n}\int_0^{\sqrt{1-|x|_2^2}} f_r\left(\dfrac{\alpha y}{h(Ax+v)^{1/s}}\right)g(0)dydx\\
& \geq 2  \int_{B^n}\int_0^{\sqrt{1-|x|_2^2}} f_r\left(\dfrac{\alpha y}{h(Ax+v)^{1/s}}\right)g(0)dydx.
\end{align*}
Since $h$ is a log-concave function, there exists a convex function $\psi$ such that
$$h(Ax+v)^{1/s}= e^{-\psi(Ax+v)/s}.$$
Then 
\begin{align*}
 f_r\left(\dfrac{\alpha y}{h(Ax+v)^{1/s}}\right) = f_r\left(\alpha ye^{\psi(Ax+v)/s}\right)= f\left(\dfrac{\alpha ye^{\psi(Ax+v)/s}-1}{1-r}\right)  
\end{align*}
and for $\alpha ye^{\psi(Ax+v)/s}\geq 1$ for every $(x,y)\in B^{n+1}, y\geq \frac{1}{2}$, we have
\begin{align*}
\bar{L}_r(A\oplus \alpha, v)& \geq  2 \int_{\frac{\sqrt{3}}{2}B^n}\int_0^{\frac{1}{2}}f\left(\dfrac{\alpha ye^{\psi(Ax+v)/s}-1}{1-r}\right) g(0)dydx\\
& \geq  2 \int_{\frac{\sqrt{3}}{2}B^n}\int_0^{\frac{1}{2}} f\left(\alpha ye^{\psi(Ax+v)/s}-1\right) g(0)dydx.    
\end{align*}

By $\fb$ and $\fd$, the function $f$ is coercive to the right, and by assumption, $\psi$ is a coercive function, hence
\begin{align*}
\lim_{||(A\oplus \alpha,v)||\rightarrow +\infty}\bar{L}_r(A\oplus \alpha, v)& \geq \lim_{||(A\oplus \alpha,v)||\rightarrow +\infty} 2 \int_{\frac{\sqrt{3}}{2}B^n}\int_0^{\frac{1}{2}} f\left(\alpha ye^{\psi(Ax+v)/s}-1\right) g(0)dydx\\
& = +\infty.
\end{align*}
\end{proof}

\begin{prop}\label{prop14.3}
Let $r\in(1/2,1)$, and assume $\gc,\gd,\fb,\fc,\fd$. The function $\bar{L}_r$ restricted to $\s\mathcal{E}_+$ is positive.
\end{prop}

\begin{proof}
First,  since $g_r(s)=0$ whenever $s>2-r$ for $r\in(1/2,1)$, then
$$g_r\left(\dfrac{|x|_2^2+y^2-1}{2h(x)^{2/s}}+1\right)=0 \quad \Leftrightarrow \quad \dfrac{|x|_2^2+y^2-1}{2h(x)^{2/s}}+1 \geq 2-r>1.$$
For $(x,y)\in B^{n+1}, y\geq 0$, it holds that
$$\dfrac{|x|_2^2+y^2-1}{2h(x)^{2/s}}+1\leq 1,$$
which implies $g_r\left(\dfrac{|x|_2^2+y^2-1}{2h(x)^{2/s}}+1\right)>0$ for all $(x,y)\in B^{n+1}$ with $y \geq 0$.

Now take $(A\oplus \alpha, v)\in \s\mathcal{E}_+$ and assume $\bar{L}_r(A\oplus \alpha, v)=0$. Since $g_r\left(\dfrac{|x|_2^2+y^2-1}{2h(x)^{2/s}}+1\right)>0$ for all $(x,y)\in B^{n+1}, y\geq 0$, then we must have $f_r\left(\dfrac{\alpha y}{h(Ax+v)^{1/s}}\right)=0$ for all $(x,y)\in B^{n+1}\cap (\mathbb{R}^n\times [0,\infty))$, which is equivalent to 
$$\dfrac{\alpha y}{h(Ax+v)^{1/s}}\leq r \quad \Leftrightarrow \quad \dfrac{\alpha y}{r}\leq h(Ax+v)^{1/s}.$$
Thus,
$$\left(Ax+v, \dfrac{\alpha y}{r}\right)\in \s\bar{h}$$
for all $(x,y)\in B^{n+1}, y\geq 0$, that is,
$$\left(A\oplus \dfrac{\alpha}{r}\right)B^{n+1}+v\subset \s \bar{h}.$$

Using that $\hslash_{B^{n+1}}$ is the John $s$-function of $h$, we  obtain
$$\s \mu\left(\left(A\oplus \dfrac{\alpha}{r}\right)B^{n+1}+v \right)= \left(\dfrac{\alpha}{r}\right)^s \det(A){}\s\mu(B^{n+1})\leq \s\mu(B^{n+1}).$$
This implies that 
$$\alpha^s\det(A)\leq r^s< 1,$$
which is a contradiction since $A\oplus \alpha\in \s\mathcal{E}_+$. 
\end{proof}

\begin{prop}\label{prop14.3.*}
Let $r\in(1/2,1)$ and assume $\gc,\gd,\fb,\fc,\fd$. Take $(A\oplus \alpha, v), (B\oplus \beta, w)\in \s\mathcal{E}_+$.  The functional $\bar{L}_r$ satisfies the property
\begin{align*}
\bar{L}_r((\lambda A+(1-\lambda)B)\oplus \alpha^\lambda\beta^{1-\lambda}, \lambda v + (1-\lambda)w)\leq \lambda\bar{L}_r(A\oplus \alpha, v) + (1-\lambda)\bar{L}_r(B\oplus \beta, w) 
\end{align*}
for all $\lambda\in [0,1]$.

We will call this property convex*.
\end{prop}

\begin{proof}
First, since  $f$ is non-decreasing and by \eqref{log}, we obtain 
$$f_r\left(\dfrac{\alpha^\lambda\beta^{1-\lambda}y}{h(\lambda(Ax+v)+(1-\lambda)(Bx+w))^{1/s}}\right)\leq f_r\left(\dfrac{\alpha^\lambda\beta^{1-\lambda}y}{h(Ax+v)^{\lambda/s}h(Bx+w)^{(1-\lambda)/s}}\right),$$
for each $\lambda\in [0,1]$. Moreover, by \eqref{media} and using that $f$ is convex, we arrive at
\begin{align*}
f_r\left(\dfrac{\alpha^\lambda\beta^{1-\lambda}y}{h(Ax+v)^{\lambda/s}h(Bx+w)^{(1-\lambda)/s}}\right) 
&\leq f_r\left(\lambda \dfrac{\alpha y }{h(Ax+v)^{1/s}}+ (1-\lambda)\dfrac{\beta y}{h(Bx+w)^{1/s}}\right)\\
& \leq \lambda f_r\left( \dfrac{\alpha y }{h(Ax+v)^{1/s}}\right)+(1-\lambda)f_r\left(\dfrac{\beta y}{h(Bx+w)^{1/s}}\right). 
\end{align*}
To finish, by inequalities \eqref{eqconvex*} and \eqref{eqconve*1}, it holds that $((\lambda A+(1-\lambda)B)\oplus \alpha^\lambda\beta^{1-\lambda}, \lambda v + (1-\lambda)w)\in\s\mathcal{E}_+$, thus
\begin{align*}
&\bar{L}_r((\lambda A+(1-\lambda)B)\oplus \alpha^\lambda\beta^{1-\lambda}, \lambda v + (1-\lambda)w)\\
& =\dfrac{1}{1-r}\int_{\mathbb{R}^n}\int_0^\infty f_r\left(\dfrac{\alpha^\lambda\beta^{1-\lambda}y}{h(\lambda(Ax+v)+(1-\lambda)(Bx+w))^{1/s}}\right)g_r\left(\dfrac{|x|_2^2+y^2-1}{2h(x)^{2/s}}+1\right)dydx\\
& \leq\lambda\bar{L}_r(A\oplus \alpha,v) + (1-\lambda)\bar{L}_r(B\oplus\beta,w), 
\end{align*}
as we wanted to prove.
\end{proof}

\begin{prop}\label{prop15.3}
Assume $\gee,\fc$,  then for $r\in(1/2,1)$ we have $\bar{L}_r(\bar{\Id}, 0)\leq C$, where $C$ is a constant depending  on $f,h,n$ and $s$.
\end{prop}

\begin{proof}
We know that
$$0< g_r(s)  \leq  1 \quad\Leftrightarrow \quad s<2-r,$$
for all $r\in (1/2,1)$.
Since
$$\dfrac{|x|_2^2+y^2-1}{2h(x)^{2/s}}+1 \leq 2-r \quad \Leftrightarrow \quad \dfrac{|x|_2^2+y^2-1}{2h(x)^{2/s}} \leq 1-r \leq \dfrac{1}{2} $$
and 
$$\dfrac{|x|_2^2+y^2}{2h(x)^{2/s}} \leq \dfrac{1}{2} \quad \Rightarrow \quad \dfrac{|x|_2^2+y^2-1}{2h(x)^{2/s}}  \leq \dfrac{1}{2},$$
then if $\bar{C}=\left\{ (x,y)\in \mathbb{R}^n\times [0,\infty):\dfrac{|x|_2^2+y^2}{h(x)^{2/s}} \leq 1 \right\}$, we have
$$\bar{L}_r(\bar{\Id},0)\leq \dfrac{1}{1-r}\int_{\mathbb{R}^n}\int_0^\infty f_r\left(\dfrac{y}{h(x)^{1/s}}\right)1_{\bar{C}}(x,y)dydx.$$
Now notice that $(x,y)\in\bar{C}$ implies 
$$0\leq \dfrac{y}{h(x)^{1/s}}\leq 1.$$
Making the substitution $\dfrac{y}{h(x)^{1/s}}=1+(1-r)t$, we get
\begin{align*}
\bar{L}_r(\bar{\Id},0)& \leq \dfrac{1}{1-r}\int_{\mathbb{R}^n}\int_0^\infty f_r\left(\dfrac{y}{h(x)^{1/s}}\right)1_{\bar{C}}(x,y)dydx\\
& \leq \int_{\mathbb{R}^n}\int_{\frac{-1}{1-r}}^0 f_r(1+(1-r)t)1_{\bar{C}}(x,(1+(1-r)t)h(x)^{1/s})h(x)^{1/s}(1+(1-r)t)dtdx\\
& = \int_{\mathbb{R}^n}\int_{\frac{-1}{1-r}}^0 f(t)1_{\bar{C}}(x,(1+(1-r)t)h(x)^{1/s})h(x)^{1/s}(1+(1-r)t)dtdx.
\end{align*}
Observe that
$$1_{\bar{C}}(x,(1+(1-r)t)h(x)^{1/s}) = 1 \Leftrightarrow \dfrac{|x|_2^2+(1+(1-r)t)^2h(x)^{2/s}}{h(x)^{2/s}}\leq 1 \Leftrightarrow \dfrac{|x|_2^2}{h(x)^{2/s}}\leq 1-(1+(1-r)t)^2. $$
Set
\begin{align*}
\bar{C}_1 &= \left\{ (x,t)\in\mathbb{R}^{n}\times[-1,0]: \dfrac{|x|_2^2}{h(x)^{2/s}}\leq 1-(1+(1-r)t)^2 \right\}
\end{align*}
and 
\begin{align*}
\bar{C}_2= \left\{ (x,t)\in\mathbb{R}^{n+1}: \dfrac{|x|_2^2}{h(x)^{2/s}}\leq 1 \right\}= \left\{ (x,t)\in\mathbb{R}^{n+1}: \dfrac{|x|_2}{h(x)^{1/s}}\leq 1 \right\}.
\end{align*}
Since $\bar{C}_1\subseteq \bar{C}_2, r\in(1/2,1)$, and $f(t)=0$ if $t< -1$, then
$$\bar{L}_r(\bar{\Id},0) \leq 
  2\int_{\mathbb{R}^n}\int_{-1}^0 f(t)1_{\bar{C}_2}(x,(1+(1-r)t)h(x)^{1/s})h(x)^{1/s}dtdx.$$

Since $h$ is a proper log-concave function, there exists a constant $\tilde{C}$ such that $h(x)^{1/s}\leq \tilde{C}$ for all $x\in\mathbb{R}^n$. Then,
$$(x, (1+(1-r)t)h(x)^{1/s})\in \bar{C}_2 \quad  \Rightarrow \quad |x|_2 \leq h(x)^{1/s}\leq \tilde{C}.$$
Therefore,
\begin{align*}
\bar{L}_r(\bar{\Id},0)& \leq 2\int_{\mathbb{R}^n}\int_{-1}^0 f(t)1_{\bar{C}_2}(x,(1+(1-r)t)h(x)^{1/s})h(x)^{1/s}dtdx\\
& \leq 2\int_{\tilde{C}\B}\int_{-1}^0 \tilde{C} f(t)dtdx\\
& = 2\tilde{C}^{n+1}\A_n(\B)\int_{-1}^0 f(t)dt\\
& \leq C.
\end{align*}
\end{proof}

\begin{proof}[Proof of Lemma \ref{lema_auxialiar}]
Let $\psi:\W_n(\mathbb{R})\oplus (0,+\infty)\RR\rightarrow \mathbb{R}$ be the function defined by 
$$ \psi(M\oplus\beta,w) = \s\mbox{det}_{n+1}\left(M\oplus\beta\right).$$ 

We know that $\s\V\NN\RR = \psi^{-1}(\{1\})$, where $c=1$ is a regular value of the differentiable map $\psi$, then by the  Lagrange multipliers, there exists a nonzero $\lambda_r$ such that
\begin{align}
\nabla \bar{L}_r(A_r\oplus \alpha_r,v_r)= \lambda_r \nabla \psi(A_r\oplus \alpha_r,v_r),\label{eq1.3}
\end{align}
where the gradients are taken with respect to the entire space $\W_n(\mathbb{R})\oplus (0,+\infty)\RR$.

Now, let $(V\oplus \alpha, w) \in T_{(A_r\oplus \alpha_r, v_r)}(\mathcal{E}_+\cap (\s\V\NN\RR))$. We have
\begin{align}
\nonumber \psi'(M\oplus \beta, v)\left[V\oplus \alpha, w\right]  &=  \beta^s \nabla \det(M)\cdot V + s\beta^{s-1}\alpha \det(M) \\
\nonumber &= (\beta^s \nabla \det(M)\oplus s\beta^{s-1} \det(M),0) \left[V\oplus \alpha, w\right].
\end{align}
Thus, since $\nabla \det(A_r) = \det({A_r})A_r^{-T}$, at the point $(A_r\oplus \alpha_r, v_r)$,  we arrive at
\begin{align}
\nonumber \nabla \psi(A_r\oplus \alpha_r, v_r) & = \left(\alpha_r^s \det(A_r)A_r^{-T}\oplus s\alpha_r^{s-1}\det(A_r),0\right)\\
\nonumber &= \alpha_r^s\det(A_r)\left(A_r^{-T}\oplus \dfrac{s}{\alpha_r},0\right)\\
\nonumber &=\left( (\Id\oplus s)\left(A_r^{-T}\oplus \dfrac{1}{\alpha_r}\right),0\right)\\
&=\left( (\Id\oplus s)\left(A_r\oplus \alpha_r\right)^{-T},0\right). \label{eq2.3}
\end{align}

We denote the function $\dfrac{\alpha y}{h(Mx+v)^{1/s}}$ by $\varphi(M,\alpha,v)$. Taking the derivative of the function $\bar{L}_r$ at the point $(M\oplus \beta,v )$ in the direction of  the vector $(V\oplus \alpha, w)$, we obtain
\begin{align*}
&\bar{L}_r'(M\oplus \beta,v ) \left[V\oplus \alpha, w\right]\\
 & =  \dfrac{1}{1-r}\int_{\mathbb{R}^n}\int_0^\infty f_r'\left(\varphi(M,\beta,v)\right)g_r\left(\dfrac{|x|_2^2+y^2-1}{2h(x)^{2/s}}+1\right) \left\langle \nabla\varphi\left(M,\beta,v\right), (V\oplus \alpha)(x,1)+(w,0)\right\rangle   dy dx\\
 &  = \dfrac{1}{1-r}\int_{\mathbb{R}^n}\int_0^\infty f_r'\left(\varphi(M,\beta,v)\right)g_r\left(\dfrac{|x|_2^2+y^2-1}{2h(x)^{2/s}}+1\right) (\left\langle \nabla\varphi\left(M,\beta,v\right)\otimes(x,1), (V\oplus \alpha)\right\rangle  \\
 &  \qquad +\langle \nabla\varphi\left(M,\beta,v\right), (w,0)\rangle) dy dx.
\end{align*}

We have
\begin{align*}
\nabla \varphi(M,\beta,v)= &\left( \dfrac{-\beta y \nabla h(Mx+v)^{1/s}}{h(Mx+v)^{2/s}}, \dfrac{y}{h(Mx+v)^{1/s}}\right),
\end{align*}
so
\begin{align*}
&\bar{L}_r'(M\oplus \beta,v ) \left[V\oplus \alpha, w\right]= \dfrac{1}{1-r}\int_{\mathbb{R}^n}\int_0^\infty f_r'\left(\varphi(M,\beta,v)\right)g_r\left(\dfrac{|x|_2^2+y^2-1}{2h(x)^{2/s}}+1\right)\\
& \quad \times \left\langle \left(
\dfrac{-\beta y \nabla h(Mx+v)^{1/s}}{h(Mx+v)^{2/s}}\otimes x \oplus \dfrac{y}{h(Mx+v)^{1/s}}, \dfrac{-\beta y \nabla h(Mx+v)^{1/s}}{h(Mx+v)^{2/s}}\right), \left(V\oplus \alpha, w\right)\right\rangle dydx.
\end{align*}
Thus,
\begin{align}
\nonumber & \nabla \bar{L}_r(A_r\oplus \alpha_r, v_r) = \dfrac{1}{1-r}\int_{\mathbb{R}^n}\int_0^\infty f_r'\left(\dfrac{\alpha_r y}{h(A_rx+v_r)^{1/s}}\right)g_r\left(\dfrac{|x|_2^2+y^2-1}{2h(x)^{2/s}}+1\right)\\
& \times \left(\left(\dfrac{-\alpha_r y \nabla h(A_rx+v_r)^{1/s}}{h(A_rx+v_r)^{2/s}}\otimes x\right) \oplus \dfrac{y}{h(A_rx+v_r)^{1/s}}
, \dfrac{-\alpha_r y \nabla h(A_rx+v_r)^{1/s}}{h(A_rx+v_r)^{2/s}} \right)dydx. \label{eq3.3}
\end{align}
Substituting  \eqref{eq2.3} and \eqref{eq3.3} into  \eqref{eq1.3} and using that $x\otimes Ay = (x\otimes y)A^T$, we obtain
\begin{align*}
& \lambda_r \left( (\Id\oplus s)\left(A_r\oplus \alpha_r\right)^{-T},0\right)  =  \dfrac{1}{1-r}\int_{\mathbb{R}^n}\int_0^\infty f_r'\left(\dfrac{\alpha_r y}{h(A_rx+v_r)^{1/s}}\right)g_r\left(\dfrac{|x|_2^2+y^2-1}{2h(x)^{2/s}}+1\right)\\
& \quad \times \left(\left(\dfrac{-\alpha_r y \nabla h(A_rx+v_r)^{1/s}}{h(A_rx+v_r)^{2/s}}\otimes x\right) \oplus \dfrac{y}{h(A_rx+v_r)^{1/s}}
, \dfrac{-\alpha_r y \nabla h(A_rx+v_r)^{1/s}}{h(A_rx+v_r)^{2/s}} \right)dydx\\
&  =  \dfrac{1}{1-r}\int_{\mathbb{R}^n}\int_0^\infty f_r'\left(\dfrac{ y}{h(x)^{1/s}}\right)g_r\left(\dfrac{|A_r^{-1}(x-v_r)|_2^2+(\alpha_r^{-1}y)^2-1}{2h(A_r^{-1}(x-v_r))^{2/s}}+1\right)\\
& \quad  \times  \left(\left(\dfrac{- y \nabla h(x)^{1/s}}{h(x)^{2/s}}\otimes A_r^{-1}(x-v_r)\right) \oplus \dfrac{y}{\alpha_rh(x)^{1/s}}
, \dfrac{- y \nabla h(x)^{1/s}}{h(x)^{2/s}} \right)\dfrac{1}{\alpha_r\det(A_r)}dydx\\
&  =  \dfrac{1}{1-r}\int_{\mathbb{R}^n}\int_0^\infty f_r'\left(\dfrac{ y}{h(x)^{1/s}}\right)g_r\left(\dfrac{|A_r^{-1}(x-v_r)|_2^2+(\alpha_r^{-1}y)^2-1}{2h(A_r^{-1}(x-v_r))^{2/s}}+1\right)\\
& \quad \times \left(\left(\dfrac{- y \nabla h(x)^{1/s}}{h(x)^{2/s}}\otimes (x-v_r)\right)A_r^{-T} \oplus \dfrac{y}{h(x)^{1/s}}\alpha_r^{-1}
, \dfrac{- y \nabla h(x)^{1/s}}{h(x)^{2/s}} \right)\alpha_r^{s-1}dydx\\
&  =  \dfrac{1}{1-r}\int_{\mathbb{R}^n}\int_0^\infty f_r'\left(\dfrac{ y}{h(x)^{1/s}}\right)g_r\left(\dfrac{|A_r^{-1}(x-v_r)|_2^2+(\alpha_r^{-1}y)^2-1}{2h(A_r^{-1}(x-v_r))^{2/s}}+1\right)\\
& \quad \times \left(\left(\dfrac{- y \nabla h(x)^{1/s}}{h(x)^{2/s}}\otimes (x-v_r)\oplus \dfrac{y}{h(x)^{1/s}}\right) (A_r\oplus\alpha_r)^{-T}
, \dfrac{- y \nabla h(x)^{1/s}}{h(x)^{2/s}} \right)\alpha_r^{s-1}dydx.
\end{align*}
Finally, using the vector equality and noting that $f_r'(s)=\dfrac{1}{1-r}(f')_r(s)$, we obtain 
\begin{align*}
\nonumber (1-r)\lambda_r(\Id\oplus s) & = \dfrac{\alpha_r^{s-1}}{1-r}\int_{\mathbb{R}^n}\int_0^\infty (f')_r\left(\dfrac{ y}{h(x)^{1/s}}\right)g_r\left(\dfrac{|A_r^{-1}(x-v_r)|_2^2+(\alpha_r^{-1}y)^2-1}{2h(A_r^{-1}(x-v_r))^{2/s}}+1\right)\\
& \qquad \times \dfrac{y}{h(x)^{3/s}}\left(- \nabla h(x)^{1/s}h(x)^{1/s}\otimes x\oplus h(x)^{1/s}h(x)^{1/s}\right)dydx
\end{align*}
and
\begin{align*}
\nonumber 0  & =  \dfrac{\alpha_r^{s-1}}{1-r}\int_{\mathbb{R}^n}\int_0^\infty (f')_r\left(\dfrac{ y}{h(x)^{1/s}}\right)g_r\left(\dfrac{|A_r^{-1}(x-v_r)|_2^2+(\alpha_r^{-1}y)^2-1}{2h(A_r^{-1}(x-v_r))^{2/s}}+1\right)\dfrac{y}{h(x)^{3/s}}\\
&  \qquad \times \left(- \nabla h(x)^{1/s}h(x)^{1/s}\right)dydx.
\end{align*}
Thus,  this concludes the proof.
\end{proof}

\section{Proof of  main results}

\begin{proof}[Proof of Theorem \ref{teo3.3}]
First, we will calculate the derivative of $\bar{I}_\nu$  at the point $(M\oplus \beta, w)\in\mathcal{E}$, in the direction of $\left(V\oplus \alpha, v\right)\in T_{(M\oplus \beta, w)}\mathcal{E}$. By  \eqref{eq} and the  inner product given by \eqref{produto}, we obtain  
\begin{align*}
&\bar{I}_{\nu}'(M\oplus \beta, w)[V\oplus \alpha, v]= \int_{\B}h(x)^{1/s}F'\left(\dfrac{\langle x,Mx+w\rangle}{h(x)^{2/s}}+\beta\right)\left(\dfrac{\langle x,Vx+v\rangle}{h(x)^{2/s}}+\alpha\right)d\nu(x)\\
& =\int_{\B}h(x)^{1/s}F'\left(\dfrac{\langle x,Mx+w\rangle}{h(x)^{2/s}}+\beta\right)\left( \left\langle \left(\dfrac{x\otimes x}{h(x)^{2/s}}, \dfrac{x}{h(x)^{2/s}}\right), (V,v)\right\rangle +\dfrac{h(x)^{1/s}h(x)^{1/s}\alpha}{h(x)^{2/s}}\right)d\nu(x)\\
&= \int_{\B}h(x)^{1/s}F'\left(\dfrac{\langle x,Mx+w\rangle}{h(x)^{2/s}}+\beta\right) \left\langle \left(\dfrac{x\otimes x\oplus h(x)^{1/s}h(x)^{1/s}}{h(x)^{2/s}}, \dfrac{x}{h(x)^{2/s}}\right), (V\oplus\alpha,v)\right\rangle d\nu(x)\\
&= \int_{\B}\dfrac{1}{h(x)^{1/s}}F'\left(\dfrac{\langle x,Mx+w\rangle}{h(x)^{2/s}}+\beta\right) \left\langle \left(x\otimes x\oplus h(x)^{1/s}h(x)^{1/s}, x \right), \left(V\oplus \alpha, v\right)\right\rangle d\nu(x).
\end{align*}
Since  $(x\otimes x\oplus h(x)^{1/s}h(x)^{1/s},x)\in \mathcal{E}$,  we conclude that
$$\nabla \bar{I}_\nu(M\oplus\beta,w) = \int_{\B}\dfrac{1}{h(x)^{1/s}}F'\left(\dfrac{\langle x,Mx+w\rangle}{h(x)^{2/s}}+\beta\right) (x\otimes x\oplus h(x)^{1/s}h(x)^{1/s},x)d\nu(x).$$

The gradient of the function $\psi\left(V\oplus \alpha, v\right)= \s\mbox{tr}\left(V\oplus \alpha\right)$ is $\nabla \psi\left(V\oplus \alpha, v\right)= (\Id\oplus s, 0)$, and  $\s\R_{n+1,0}(\mathbb{R})\RR$ is the orthogonal complement of $(\Id\oplus s, 0)$ in $\mathcal{E}$. Since $(M_0\oplus \beta_0,w_0)\in \s\R_{n+1,0}(\mathbb{R})\RR$ is a singular point of  $\bar{I}_\nu$ and 0 is a regular value of $\psi$,  by the Lagrange Multiplier Theorem, there exists $\lambda\in\mathbb{R}$ such that 
$$\nabla \bar{I}_\nu(M_0\oplus \beta_0,w_0)=\lambda \nabla\psi(M_0\oplus \beta_0,w_0),$$ 
that is,
\begin{align}\label{eq9.3}
\int_{\B}\dfrac{1}{h(x)^{1/s}}F'\left(\dfrac{\langle x,M_0x+w_0\rangle}{h(x)^{2/s}}+\beta_0\right) (x\otimes x\oplus h(x)^{1/s}h(x)^{1/s})d\nu(x)&  =\lambda(\Id\oplus s) \\
\nonumber \int_{\B}\dfrac{1}{h(x)^{1/s}}F'\left(\dfrac{\langle x,M_0x+w_0\rangle}{h(x)^{2/s}}+\beta_0\right)xd\nu(x)&= 0.
\end{align}

To prove that $\lambda$ is positive,  recall that $F$ is non-decreasing, and hence 
$F'\left(\dfrac{\langle x,M_0x+w_0\rangle}{h(x)^{2/s}}+\beta_0\right)\geq 0$. Taking the trace function in  equation \eqref{eq9.3} and recalling that the support of measure $\nu$ is a subset of  points of $\B$ where $h(x)^{1/s} = \hslash_{B^{n+1}}(x)$, we arrive at 
\begin{align*}
\lambda & = \dfrac{1}{n+s}\int_{\B}\dfrac{1}{h(x)^{1/s}}F'\left(\dfrac{\langle x,M_0x+w_0\rangle}{h(x)^{2/s}}+\beta_0\right)d\nu(x).
\end{align*}
By Theorem \ref{teo6.3}, we know that $\dfrac{\langle x,M_0x+w_0\rangle}{h(x)^{2/s}}+\beta_0>0$ holds for a set of positive $\nu$-measure. Since $F'(x)\geq 0$ for all $x$ and $F'(x)>0$ for $x\geq 0$, we conclude that $\lambda>0$, and the proof is complete.
\end{proof}

\begin{lema}\label{=1}
If $(A_r\oplus \alpha_r, v_r)\in \s\mathcal{E}_+$ minimizes $\bar{L}_r$, then $\s\mbox{det}_{n+1}(A_r\oplus \alpha_r)=1$.   
\end{lema}

\begin{proof}
Assume that $\s\mbox{det}_{n+1}(A_r\oplus \alpha_r)>1$, that is, $\alpha_r^s\det(A_r)>1$. Take $\bar{A}_r= A_r\oplus \dfrac{1}{\det(A_r)^{1/s}}$. Then, $(\bar{A}_r,v_r)\in \s\mathcal{E}_+\cap(\s\V\NN\RR) $.  Notice that since $B^{n+1}$ is the John $s$-ellipsoid  of $h$ and $\s\mbox{det}_{n+1}(\bar{A})\geq 1$, then
$$(\bar{A}_rB^{n+1}+v_r)\setminus \s\bar{r^sh}$$
must have non-empty interior. In fact,  to say that $(\bar{A}_rB^{n+1}+v_r)\setminus \s\bar{r^sh}$ has  empty interior is the same as to say that  $ \bar{A}_rB^{n+1}+v_r \subseteq (\Id\oplus r)\s\bar{h}$. But 
\begin{align*}
\s\mbox{det}_{n+1}\left((\Id\oplus r)^{-1}\bar{A}_r\right)&= {}\s\mbox{det}_{n+1}\left(A_r\oplus \frac{1}{r\det(A_r)^{1/s}}\right)
=  \frac{1}{r^s}>1,
\end{align*}
which contradicts the fact that  $B^{n+1}$ is the John $s$-function of $h$. Since $(\bar{A}_rB^{n+1}+v_r)\setminus \s\bar{r^sh}$ has non-empty interior,  there exists a subset $\bar{C}$ of $B^{n+1}$ such that $\A_{n+1}(\bar{C})>0$, and  for every $(x,y)\in\bar{C}$,  the following inequality holds
$$r h\left(A_r x+v_r\right)^{1/s}< \dfrac{y}{\det(A_r)^{1/s}}.$$
Hence, this implies that $f_r\left(\dfrac{y}{\det(A_r)^{1/s}h(A_rx+v_r)^{1/s}}\right)$ is positive. Moreover, in this set $\bar{C}$,  it holds that $g_r\left(\dfrac{|x|_2^2+y^2-1}{2h(x)^{2/s}}+1\right)$ is positive since $\dfrac{|x|_2^2+y^2-1}{2h(x)^{2/s}}\leq 0$. Thus,
\begin{align*}
\bar{L}_r(\bar{A}_r,v_r) &= \dfrac{1}{1-r}\int_{\mathbb{R}^n}\int_0^{\infty}f_r\left(\dfrac{y}{\det(A_r)^{1/s}h(A_rx+v_r)^{1/s}}\right)g_r\left(\dfrac{|x|_2^2+y^2-1}{2h(x)^{2/s}}+1\right)dydx\\
& < \dfrac{1}{1-r}\int_{\mathbb{R}^n}\int_0^{\infty}f_r\left(\dfrac{\alpha_r y}{h(A_rx+v_r)^{1/s}}\right)g_r\left(\dfrac{|x|_2^2+y^2-1}{2h(x)^{2/s}}+1\right)dydx\\
& = \bar{L}_r(A_r\oplus \alpha_r, v_r),
\end{align*}
which contradicts the minimality of $(A_r\oplus\alpha_r, v_r)$. Therefore, $\s\mbox{det}_{n+1}(A_r\oplus \alpha_r)=1.$
\end{proof}

\begin{proof}[Proof of Theorem \ref{teo7.3}]
The existence of a minimum of the functional $\bar{L}_r$ follows from the fact that it is coercive in $\s\mathcal{E}_+$ and  this set is a closed convex set. Now we assume that there are two distinct minimum of $\bar{L}_r$ in $\s\mathcal{E}_+$, say $(A\oplus\alpha,v)$ and $(B\oplus\beta, w)$. Then, by Proposition \ref{prop14.3.*}, it holds that
$$\bar{L}_r((\lambda A+(1-\lambda)B)\oplus \alpha^\lambda\beta^{1-\lambda}, \lambda v+(1-\lambda)w)= \lambda \bar{L}_r(A\oplus \alpha,v)+(1-\lambda)\bar{L}_r(B\oplus\beta,w),$$
that is, $((\lambda A+(1-\lambda)B)\oplus \alpha^\lambda\beta^{1-\lambda}, \lambda v+(1-\lambda)w)\in\s\mathcal{E}_+$ also minimizes the functional $\bar{L}_r$ and by Lemma \ref{=1}, we have 
$$\s\mbox{det}_{n+1}((\lambda A+(1-\lambda)B)\oplus \alpha^\lambda\beta^{1-\lambda})=1.$$
By  \eqref{det}, we get
\begin{align*}
1= \s\mbox{det}_{n+1}((\lambda A+(1-\lambda)B)\oplus \alpha^\lambda\beta^{1-\lambda}) & = (\alpha^s)^\lambda(\beta^s)^{1-\lambda}\det(\lambda A+ (1-\lambda)B)\\
& \geq (\alpha^s)^\lambda(\beta^s)^{1-\lambda} \det(A)^\lambda\det(B)^{1-\lambda}\\
&= (\alpha^s \det(A))^\lambda(\beta^s \det(B))^{1-\lambda}\\
& = 1. 
\end{align*}
This last equality implies that $\det(\lambda A+ (1-\lambda)B) = \det(A)^\lambda\det(B)^{1-\lambda}$, and hence we  have $A=B$. Since 
\begin{align*}
\beta^s \det(B)= 1 = \alpha^s\det(A),
\end{align*}
it follows that $\alpha=\beta$. Then, up to horizontal translation, the minimizers $(A\oplus\alpha,v)$ and $(B\oplus\beta, w)$ coincide.

Denote $M_r\oplus\beta_r = \frac{A_r\oplus \alpha_r-\bar{\Id}}{1-r}, w_r=\frac{v_r}{1-r}$. Since $(A_r\oplus \alpha_r, v_r)\in \s\mathcal{E}_+\cap(\s\V\N\RR)$, we have
$$\bar{L}_r(A_r\oplus\alpha_r,v_r)=\alpha_r^{s-1}\bar{I}_r(M_r\oplus \beta_r,w_r),$$
and $(M_r\oplus \beta_r,w_r)$ is, up to horizontal translation, the  unique global minimum of the restriction of $\bar{I}_r$ to 
$$\dfrac{\s\mathcal{E}_+\cap (\s\V\NN\RR)-\bar{\Id}\RR}{1-r}.$$

Our next step is to prove that  $(A_r\oplus\alpha_r,v_r)\rightarrow (\bar{\Id},0)$. Assume that $(A_r\oplus\alpha_r,v_r)$ does not converge to $(\bar{\Id},0)$. Since by Propositions \ref{prop13.3}  and \ref{prop15.3}, the sequence $\{(A_r\oplus\alpha_r,v_r)\}_r$ is bounded, there exists a sequence $r_k\rightarrow 1^-$ such that $\{(A_{r_k}\oplus\alpha_{r_k},v_{r_k})\}_{k}$ converges. Assume that $(A_{r_k}\oplus\alpha_{r_k},v_{r_k}) \rightarrow (A^*\oplus \alpha^*, v^*)\in  \s\mathcal{E}_+\cap(\s\V\NN\RR)$ with $(A^*\oplus \alpha^*, v^*)\neq (\bar{\Id},0)$. Again, since $B^{n+1}$ is the John $s$-ellipsoid of $h$ and $\s\mbox{det}_{n+1}(A^*\oplus\alpha^*)=1$, then the set $((A^*\oplus \alpha^*)B^{n+1}+v^*)\setminus \s\bar{h}$ has  positive Lebesgue measure. Take $\rho <1$ such that the set  $(\rho(A^*\oplus \alpha^*)B^{n+1}+v^*)\setminus \s\bar{h}$ has positive Lebesgue measure. For
large $k$, we have $\rho(A^*\oplus\alpha^*)B^{n+1}+v^* \subseteq (A_{r_k}\oplus \alpha_{r_k})B^{n+1}+v_{r_k}$. By Fatou's lemma,
\begin{align*}
&\liminf_{k\rightarrow +\infty}\bar{L}_{r_k}(A_{r_k}\oplus\alpha_{r_k},v_{r_k})  \\
& = \liminf_{k\rightarrow +\infty} \dfrac{1}{1-r_k}\int_{\mathbb{R}^n}\int_0^\infty f_{r_k}\left(\dfrac{\alpha_{r_k}y}{h(A_{r_k}x+v_{r_k})^{1/s}}\right)g_{r_k}\left(\dfrac{|x|_2^2+y^2-1}{2h(x)^{2/s}}+1\right)dydx\\
& \geq \liminf_{k\rightarrow +\infty} \dfrac{\alpha_{r_k}^{s-1}}{1-r_k}\int_{\mathbb{R}^n\setminus \s\bar{h}}\int_0^\infty f_{r_k}\left(\dfrac{y}{h(x)^{1/s}}\right)g_{r_k}\left(\dfrac{|A_{r_k}^{-1}(x-v_{r_k})|_2^2+(\alpha_{r_k}^{-1}y)^2-1}{2h(A_{r_k}^{-1}(x-v_{r_k}))^{2/s}}+1\right)dydx. 
\end{align*}
Notice that if $(\tilde{x},\tilde{y})\in B^{n+1}$ and $(x,y)= \rho(A^*\oplus\alpha^*)(\tilde{x},\tilde{y})+v^*$, then
$$(A_{r_k}^{-1}\oplus \alpha_{r_k}^{-1})(\rho(A^*\oplus\alpha^*)(\tilde{x},\tilde{y})+v^*-v_{r_k})\in (A_{r_k}\oplus \alpha_{r_k})^{-1}(A_{r_k}\oplus\alpha_{r_k})B^{n+1} = B^{n+1},$$
from which $|A_{r_k}^{-1}(x-v_{r_k})|_2^2+(\alpha_{r_k}^{-1}y)^2\leq 1$. And by $\gb$ it follows that
$$g_{r_k}\left(\dfrac{|A_{r_k}^{-1}(x-v_{r_k})|_2^2+(\alpha_{r_k}^{-1}y)^2-1}{2h(A_{r_k}^{-1}(x-v_{r_k}))}+1\right)\geq g_{r_k}(1)=g(0).$$
Thus,
\begin{align*}
\liminf_{k\rightarrow +\infty}\bar{L}_{r_k}(A_{r_k}\oplus\alpha_{r_k},v_{r_k}) \geq &\liminf_{k\rightarrow +\infty} \dfrac{\alpha_{r_k}^{s-1}}{1-r_k}\int_{(\rho(A^*\oplus\alpha^*)B^{n+1}+v^*) \setminus\s\bar{h}}\int_0^\infty f_{r_k}\left(\dfrac{y}{h(x)^{1/s}}\right)g(0)dydx\\
=& \infty,
\end{align*}
which contradicts the boundedness of the minimizer $(A_r\oplus\alpha_r,v_r)$, since by Proposition \ref{prop15.3}
$$\bar{L}_{r_k}(A_{r_k}\oplus\alpha_{r_k},v_{r_k}) \leq \bar{L}_r(\bar{\Id},0)\leq C.$$

To finish, we need to prove that $\s\mbox{tr}\left(\dfrac{M_r\oplus \beta_r}{||M_r\oplus ||_F}\right)\rightarrow 0$.  A simple calculate shows that  $\s$tr is the differential of $\s\mbox{det}_{n+1}$ at $\bar{\Id}\in\W\NN$. By Taylor expansion,  we have
$$\s\mbox{det}_{n+1}(\bar{\Id}+\bar{V})= 1+\langle \Id\oplus s,\bar{V}\rangle+o(||\bar{V}||_F)=1+\s\mbox{tr}(\bar{V})+ o(||\bar{V}||_F),$$
where $\frac{o(\epsilon)}{\epsilon}\rightarrow 0$ as $\epsilon\rightarrow 0$. Taking $\bar{V}=(1-r)(M_r\oplus\beta_r)$, we get
\begin{align*}
1 &= \s\mbox{det}_{n+1}(A_r\oplus\alpha_r)\\
  &= \s\mbox{det}_{n+1}(\bar{\Id}+(1-r)(M_r\oplus\beta_r))\\
  &= \ 1+(1-r)\s\mbox{tr}(M_r\oplus\beta_r)+o((1-r)||M_r\oplus\beta_r||_F).
\end{align*}
Therefore,
$$\s\mbox{tr}\left(\dfrac{M_r\oplus\beta_r}{||M_r\oplus\beta_r||_F}\right)= \dfrac{\s\mbox{tr}(M_r\oplus\beta_r)}{||M_r\oplus\beta_r||_F}=-\dfrac{o((1-r)||M_r\oplus\beta_r||_F)}{(1-r)||M_r\oplus\beta_r||_F}\rightarrow 0$$
as $r\rightarrow 1^-$.    
\end{proof}

\begin{proof}[Proof of Theorem \ref{teo8.3}]
Let us denote by   $o((1-r)^a)$ (resp. $o(1)$)  any function of the involved parameters $M,\beta, w, r,s, t,x$, satisfying 
$$\lim_{r\rightarrow1^-}\dfrac{o((1-r))^a}{(1-r)^a}=0\left(\mbox{resp. } \lim_{r\rightarrow 1^-}o(1)=0\right),$$
where the limits are uniform in compact sets with respect to the parameters. Similarly, $O(1)$  denotes any bounded function.

By Taylor expansion, we have the following expression for all $x,w\in\mathbb{R}^n$ and $M\oplus \beta\in \bar{B}_r$ (recall that $\bar{B}_r$ is the domain of the functional $\bar{I}_r$)
\begin{align*}
|(\Id+(1-r)M)^{-1}(x-(1-r)w)|_2  &= |x-(1-r)(Mx+w)+o(1-r)|_2\\
&=  |x|_2-(1-r)\left\langle\dfrac{x}{|x|_2}, Mx+w\right\rangle+o(1-r).
\end{align*}
For short, we denote
\begin{align*}
& \psi(M;\beta; w; (x,t))\\
&= \dfrac{|(\Id+(1-r)M)^{-1}(x-(1-r)w)|_2^2 +((1+(1-r)\beta)^{-1}(1+(1-r)t)h(x)^{1/s})^2-1}{2h((\Id+(1-r)M)^{-1}(x-(1-r)w))^{2/s}}+1.    
\end{align*}
Since
\begin{align*}
\bar{I}_r(M\oplus\beta,w)& = \dfrac{1}{1-r}\int_{\mathbb{R}^n}\int_0^\infty f_r\left(\dfrac{y}{h(x)^{1/s}}\right)\\
& \times g_r\left(\dfrac{|(\Id+(1-r)M)^{-1}(x-(1-r)w)|_2^2 +((1+(1-r)\beta)^{-1}y)^2-1}{2h((\Id+(1-r)M)^{-1}(x-(1-r)w))^{2/s}}+1\right)dydx,
\end{align*}
substituting  $\frac{y}{h(x)^{1/s}}=1+(1-r)t$, we obtain
\begin{align*}
\bar{I}_r(&M\oplus\beta,w)=\int_{\mathbb{R}^n}\int_{-\frac{1}{1-r}}^\infty h(x)^{1/s}(1+(1-r)t)f_r\left(1+(1-r)t\right)g_r\left(\psi(M;\beta; w; (x,t))\right)dtdx.
\end{align*}
To calculate $\psi(M;\beta; w; (x,t))$, note the following expansions
\begin{align*}
 \bullet \ \ & |(\Id+(1-r)M)^{-1}(x-(1-r)w)|_2^2 = |x|_2^2+(1-r)^2\left\langle\dfrac{x}{|x|_2}, Mx+w\right\rangle^2\\
 &-2|x|_2(1-r)\left\langle\dfrac{x}{|x|_2}, Mx+w\right\rangle+o(1-r)^2 +2o(1-r)\left(|x|_2-(1-r)\left\langle\dfrac{x}{|x|_2}, Mx+w\right\rangle\right);\\
\bullet \ \  & ((1+(1-r)\beta)^{-1}h(x)^{1/s}(1+(1-r)t))^2 = \dfrac{h(x)^{2/s}(1+2(1-r)t+(1-r)^2t^2)}{(1+(1-r)\beta)^2};\\
\bullet \ \ & 2h((\Id+(1-r)M)^{-1}(x-(1-r)w))^{2/s}=2h(x-(1-r)(Mx+w)+o(1-r))^{2/s}.
\end{align*}
Thus, we obtain the following expression for $ \psi(M; \beta; w; (x,t)) $
\begin{align*}
\psi(M;\beta; w; (x,t)) &= \dfrac{(|x|_2^2+h(x)^{2/s}-1) + (2\beta(|x|_2^2-1)(1-r)}{2(1+(1-r)\beta)^2h(x-(1-r)(Mx+w)+o(1-r))^{2/s}(1-r)}\\
& \quad + \dfrac{(1-r)\left\langle\dfrac{x}{|x|_2}, Mx+w\right\rangle -2\left\langle x, Mx+w\right\rangle}{2h(x-(1-r)(Mx+w)+o(1-r))^{2/s}}\\ 
&\quad  +\dfrac{2h(x)^{2/s}t + (1-r)h(x)^{2/s}t^2 + (1-r)\beta^2(|x|_2^2-1)}{2(1+(1-r)\beta)^2h(x-(1-r)(Mx+w)+o(1-r))^{2/s}}\\
& \quad + \left[\dfrac{o(1-r)^2}{1-r}+\dfrac{2o(1-r)}{1-r}\left(|x|_2-(1-r)\left\langle \dfrac{x}{|x|_2}, Mx+w\right\rangle\right)\right]\\
& \qquad \times \dfrac{1}{2h(x-(1-r)(Mx+w)+o(1-r))^{2/s}}. 
\end{align*}

Consider the following sets:
\begin{align*}
\Lambda & =\{x\in\mathbb{R}^n: |x|_2^2+h(x)^{2/s}-1\leq 0\}\\
\Lambda^c & =\{x\in\mathbb{R}^n: |x|_2^2+h(x)^{2/s}-1>0\}.  
\end{align*}

Note that
\begin{itemize}
\item [$(i)$] If $x\in \Lambda^c$,  since $h$ is bounded and $(1+(1-r)\beta)\leq (1+\beta)$, we have
$$\dfrac{|x|_2^2+h(x)^{2/s}-1}{2(1+(1-r)\beta)^2h(x-(1-r)(Mx+w)+o(1-r))^{2/s}(1-r)}\longrightarrow +\infty$$
as $r\rightarrow 1^-$. Thus, by $\gee$, it holds that $g\mid_{\Lambda^c}\stackrel{r\rightarrow 1^-}{\longrightarrow}$0;

\item [$(ii)$] If $x\in \Lambda$, then $|x|_2^2+h(x)^{2/s}=1$. Indeed,
$$|x|_2^2+h(x)^{2/s}<1\Leftrightarrow \sqrt{|x|_2^2+h(x)^{2/s}}< 1 \Leftrightarrow (x,h(x)^{1/s})\in \mbox{int }B^{n+1}\subset \mbox{int }\s\bar{h}.$$

\noindent But, as we know $(x,h(x)^{1/s})\in \partial\s\bar{h}$, for all $x \in \mathbb{R}^n$. Hence,
$$\Lambda=\{x\in\mathbb{R}^n: |x|_2^2+h(x)^{2/s}-1\leq 0\}=\{x\in\mathbb{R}^n: |x|_2^2+h(x)^{2/s}-1= 0\};$$ 

\item [$(iii)$] $h(x-(1-r)(Mx+w)+o(1-r))^{2/s} \stackrel{r\rightarrow 1^-}{\longrightarrow}h(x)^{2/s}$ since $h$ is  continuous.
\end{itemize}

By $\fc$, the integrand is 0 for $t<-1$ and by $(ii)$, we obtain 
\begin{align}
\nonumber & \bar{I}_r(M\oplus\beta,w) = \int_{\mathbb{R}^n}\int_{-1}^\infty h(x)^{1/s}(1+(1-r)t)f(t) g\left(\psi(M;\beta; w; (x,t))\right)dtdx\\
\nonumber &= \int_{\Lambda^c}\int_{-1}^\infty h(x)^{1/s}(1+(1-r)t)f(t) g\left(\psi(M;\beta; w; (x,t))\right)dtdx\\
\nonumber &\quad +  \int_{\Lambda}\int_{-1}^\infty h(x)^{1/s}(1+(1-r)t)f(t) g\left(\psi(M;\beta; w; (x,t))\right)dtdx\\
\nonumber &= \int_{\Lambda^c}\int_{-1}^\infty h(x)^{1/s}(1+(1-r)t)f(t)\\
\nonumber & \qquad \times g\left(  \dfrac{|x|_2^2+h(x)^{2/s}-1 + (1-r)O(1)+ (1-r)t (2h(x)^{2/s}+o(1)) +o(1)}{2(1+(1-r)\beta)^2h(x-(1-r)(Mx+w)+o(1-r))^{2/s}(1-r)}\right)dtdx\\
\nonumber & \quad + \int_{\Lambda}\int_{-1}^\infty h(x)^{1/s}(1+(1-r)t)f(t) \\
& \qquad \times g\left(  \dfrac{-2\beta(1-|x|_2^2+o(1))-2\langle x,Mx+w +o(1)\rangle +t(2h(x)^{2/s}+o(1))+o(1)}{2(1+(1-r)\beta)^2h(x-(1-r)(Mx+w)+o(1-r))^{2/s}}\right)dtdx.\label{eq85}
\end{align}
To prove that $\bar{I}_r$ converges to $\bar{I}_1$, when $r\rightarrow 1^-$, in compact sets, consider a convergent sequence $(M_k\oplus\beta_k,w_k)\rightarrow (M\oplus\beta,w)$ and $r_k\rightarrow 1^-$. By $(i)$ and $\gee$, the function $g$ in the  first integral is zero for $t>C$ where $C$ is independent of $k$. Since the functions  $f,g$ are thus uniformly bounded in the support of both integrals and it holds $(iii)$,   we may apply the
Dominated Convergence Theorem in \eqref{eq85} to obtain 
\begin{align*}
\lim_{k\rightarrow \infty}\bar{I}_{r_k}(M_k\oplus\beta_k, w_k)
&= \int_{\Lambda}\int_{-1}^\infty h(x)^{1/s}f(t)g\left(t -\dfrac{\langle x,Mx+w\rangle}{h(x)^{2/s}}-\beta\right)dtdx\\
&= \int_{\Lambda} h(x)^{1/s}\int_{-1}^\infty f(t)g\left(t -\dfrac{\langle x,Mx+w\rangle}{h(x)^{2/s}}-\beta\right)dtdx.
\end{align*}
Thus, we conclude
$$\bar{I}_1(M \oplus \beta, w) = \int_{\Lambda} h(x)^{1/s} F\left( \frac{\langle x, Mx + w \rangle}{h(x)^{2/s}} + \beta \right)  dx, $$
where \(F(x) = f \ast \bar{g}(x)\) is the convolution of \(f\) and \(\bar{g}(x) = g(-x)\).

Finally, we show that $F$ satisfies the conditions of  Theorem \ref{teo3.3}. First, $F$ is non-negative because $f(t)g(t-x)\geq 0$ for all $(x,t)\in\mathbb{R}^n\times\mathbb{R}$. Second, $F$ is non-decreasing since  both $f$ and $\bar{g}$  are non-decreasing. Specifically, we have
\begin{align*}
F'(x)&= -\int_{-\infty}^\infty f(t)g'(t-x)dt=  \int_{-\infty}^\infty f(x-t)\bar{g}'(t)dt\geq 0.
\end{align*}
By assumptions $\fa$ and $\ga$, $f$ and $g$ are locally Lipschitz, and thus absolutely continuous and differentiable almost everywhere.  Therefore, $F$ is  twice differentiable almost everywhere, and by $\fd,\gc,\gd,\gee$,
$$F''(x)= \int_{-1}^{1}f'(x-t)\bar{g}'(t)dt\geq 0,$$
which shows that $F$ is convex. To establish strict convexity on $[0,\infty)$, take any $x>0$. If $F''(x)=0$, then since $\bar{g}'(t)>0$ on $(-1,1)$, the last inequality implies that $f'=0$ on a set of positive measure inside $(x-1,x+1)$,  which contradicts $\fd$.
\end{proof}

In order to prove  Theorem \ref{teo9.3}, we before need to prove that the family of minimizers of the functionals $\bar{I}_r$ admits a convergent subsequence. 

\begin{lema}\label{bounded}
For every $r\in(1/2,1)$, let $(M_r\oplus\beta_r,w_r)$ be a minimizer of the functional $\bar{I}_r$ as given by Theorem \ref{teo7.3}. The sequence $\{(M_r\oplus\beta_r,w_r)\}_r$ is bounded. 
\end{lema}
\begin{proof}
By Lemma \ref{unico_minimo}, the functional $\bar{I}_1$ is coercive. Therefore, there exists a constant $R>0$ such that for any  $(M\oplus\beta,w)\in\s\R_{n+1,0}(\mathbb{R})\RR$, if  $||(M\oplus\beta,w)||\geq R$, then 
$$\bar{I}_1(M\oplus\beta,w)\geq C+2,$$
where $C\geq \bar{L}_r(\bar{\Id},0)$ is given by Proposition \ref{prop15.3}.

Let $\bar{B}_{2R}=\{(M\oplus\beta,w)\in\R\NN\RR: ||(M\oplus\beta,w)||\leq 2R \}$. By Theorem \ref{teo8.3}, there exists $r_0\in (1/2,1)$ such that for every $r\in(r_0,1)$ and $(M\oplus\beta,w)\in\bar{B}_{2R},$
$$|\bar{I}_r(M\oplus\beta,w)-\bar{I}_1(M\oplus\beta,w)|\leq 1/2.$$

We now  show  that  for every $r\in(r_0,1)$, $(M_r\oplus\beta_r,w_r)\in\bar{B}_{2R}$. Assume by contradiction that there exists $r\in(r_0,1)$ such that  $(M_r\oplus\beta_r,w_r)\not\in \bar{B}_{2R}$. Then, there exists $\lambda <1$ such that
$$||\lambda(M_r\oplus\beta_r,w_r)||=2R.$$ 
By \eqref{media} and since $\left.\frac{\partial }{\partial t}(1+t\beta_r)^\lambda\right|_{t=0}=\lambda\beta_r$, it  holds that for $t\geq 0$, $(1+t\beta_r)^\lambda\leq 1+t\lambda \beta_r$. Thus, for $r\rightarrow 1^-$, we have 
$$R\leq \rho=\left|\left|\left(\lambda M_r\oplus\left(\dfrac{(1+(1-r)\beta_r)^\lambda-1}{1-r}\right), \lambda w_r\right)\right|\right|\leq ||\lambda(M_r\oplus\beta_r,w_r)||=2R.$$

Since $\bar{I}_1$ is continuous on the compact set $\bar{B}_{2R}$, there is $\varepsilon>0$ such that  
$$\bar{I}_1(M\oplus\beta,w)\geq C+1$$
for every $(M\oplus\beta,w)\in\partial \bar{B}_{\rho}=\{(M\oplus\beta,w)\in\R\NN\RR: ||(M\oplus\beta,w)||= \rho, R\leq \rho\leq 2 R \}$ with $\s\mbox{tr}(M\oplus\beta)<\varepsilon$. 

By Theorem \ref{teo7.3},  it holds that  $A_r\oplus\alpha_r\rightarrow \bar{\Id}$ as $r\rightarrow 1^-$,  then increasing $r_0$ if necessary, we may assume for every $r\in(r_0,1)$ and $\lambda\in[0,1]$,
$$\mbox{det}_{n+1}(\lambda(A_r\oplus\alpha_r)+(1-\lambda)\bar{\Id})\leq \dfrac{C+1/2}{C+1/4}=1+\dfrac{1}{4C+1}$$
and, again by Theorem \ref{teo7.3}, we have that $\left|\s\mbox{tr}\left(\frac{M_r\oplus \beta_r}{||M_r\oplus\beta_r||_F}\right)\right|\leq \frac{\varepsilon}{2R}$.

Moreover,
$$\left|\s\mbox{tr}\left(\lambda M_r\oplus\left(\dfrac{(1+(1-r)\beta_r)^\lambda-1}{1-r}\right)\right)\right| \leq 
|\s\mbox{tr}(\lambda(M_r\oplus\beta_r))|\leq \dfrac{||\lambda(M_r\oplus\beta_r)||_F}{2R}\varepsilon\leq \varepsilon,$$
then we  obtain
\begin{align*}
\bar{I}_r\left(\lambda M_r\oplus\left(\dfrac{(1+(1-r)\beta_r)^\lambda-1}{1-r}\right), \lambda w_r\right)& \geq  \bar{I}_1\left(\lambda M_r\oplus\left(\dfrac{(1+(1-r)\beta_r)^\lambda-1}{1-r}\right), \lambda w_r\right)-1/2\\
&  \geq C+1/2. 
\end{align*}

Now, using that $(M_r\oplus\beta_r,w_r)=\left(\dfrac{A_r\oplus\alpha_r-\bar{\Id}}{1-r},\dfrac{v_r}{1-r}\right)$, we have
$$(((1-r)\lambda M_r+\Id)\oplus (1+(1-r)\beta)^\lambda,(1-r) \lambda w_r)= ((\lambda A_r+(1-\lambda)\Id)\oplus\alpha_r^\lambda, \lambda v_r),$$
and since $\lambda A_r+(1-\lambda)\Id \in\R\Y$ for $r\rightarrow 1^-$, we obtain
\begin{align*}
\s\mbox{det}_{n+1}((\lambda A_r+(1-\lambda)\Id)\oplus\alpha_r^\lambda)& = (\alpha_r^s)^\lambda\det(\lambda A_r+(1-\lambda)\Id)\\
& \geq  (\alpha_r^s)^\lambda \det(A_r)^\lambda\det(\Id)^{1-\lambda}\\
&  = (\s\mbox{det}_{n+1}(A_r\oplus\alpha_r))^\lambda \ {}\s\mbox{det}_{n+1}(\bar{\Id})\\
& \geq  1.
\end{align*}
Hence $((\lambda A_r+(1-\lambda)\Id)\oplus\alpha_r^\lambda,\lambda v_r)\in\s\mathcal{E}_+$ and
\begin{align*}
&\bar{L}_r((\lambda A_r+(1-\lambda)\Id)\oplus \alpha_r^\lambda,\lambda v_r)\\
&=\dfrac{1}{1-r}\int_{\mathbb{R}^n}\int_0^\infty f_r\left(\dfrac{\alpha_r^\lambda y}{h((\lambda A_r+(1-\lambda)\Id)x+\lambda v_r)^{1/s}}\right)g_r\left(\dfrac{|x|_2^2+y^2-1}{2h(x)^{2/s}}+1\right)dydx\\
& =  \dfrac{1}{1-r}\int_{\mathbb{R}^n}\int_0^\infty \dfrac{1}{\alpha_r^\lambda\det(\lambda A_r+(1-\lambda)\Id)}f_r\left(\dfrac{y}{h(x)^{1/s}}\right) \\
& \quad \times g_r\left(\dfrac{|(\lambda A_r+(1-\lambda)\Id)^{-1}(x-\lambda v_r)|_2^2+(\alpha_r^{-\lambda}y)^2-1}{2h((\lambda A_r + (1-\lambda)\Id)^{-1}(x-\lambda v_r))^{2/s}}+1\right)dydx\\
& =  \dfrac{1}{1-r}\int_{\mathbb{R}^n}\int_0^\infty \dfrac{1}{\alpha_r^\lambda\det(\lambda A_r+(1-\lambda)\Id)}f_r\left(\dfrac{y}{h(x)^{1/s}}\right)\\
& \quad \times g_r\left(\dfrac{|(\Id+(1-r)\lambda M_r)^{-1}(x-(1-r)\lambda w_r)|_2^2+((1+(1-r)\beta_r)^\lambda)^{-1}y)^2-1}{2h((\Id+(1-r)\lambda M_r)^{-1}(x-(1-r)\lambda w_r))^{2/s}}+1\right)dydx.
\end{align*}
A simple calculation using  \eqref{media}  shows the inequality
\begin{align*}
\alpha_r^\lambda \det(\lambda A_r+(1-\lambda)\Id)\leq \mbox{det}_{n+1}(\lambda (A_r\oplus\alpha_r)+(1-\lambda)\bar{\Id}),
\end{align*}
and thus
\begin{align*}
& \bar{L}_r((\lambda A_r+(1-\lambda)\Id)\oplus \alpha_r^\lambda,\lambda v_r)\geq  \dfrac{1}{1-r}\int_{\mathbb{R}^n}\int_0^\infty \dfrac{1}{\mbox{det}_{n+1}(\lambda (A_r\oplus\alpha_r)+(1-\lambda)\bar{\Id})}f_r\left(\dfrac{y}{h(x)^{1/s}}\right)\\
& \qquad  \times g_r\left(\frac{|(\Id+(1-r)\lambda M_r)^{-1}(x-(1-r)\lambda w_r)|_2^2+((1+(1-r)\beta_r)^\lambda)^{-1}y)^2-1}{2h((\Id+(1-r)\lambda M_r)^{-1}(x-(1-r)\lambda w_r))^{2/s}}+1\right)dydx\\
& =  \frac{\bar{I}_r\left(\lambda M_r\oplus\left(\dfrac{(1+(1-r)\beta_r)^\lambda-1}{1-r}\right), \lambda w_r\right)}{\mbox{det}_{n+1}(\lambda(A_r\oplus\alpha_r)+(1-\lambda)\bar{\Id})}\\
& \geq  \left(\dfrac{C+1/2}{C+1/4}\right)^{-1}(C+1/2)\\
&  \geq\bar{L}_r(\bar{\Id},0)+1/4. 
\end{align*}

Since $\bar{L}_r(\bar{\Id},0)\geq \bar{L}_r(A_r\oplus\alpha_r,v_r)$, we obtain the inequalities
$$\bar{L}_r((\lambda A_r+(1-\lambda)\Id)\oplus\alpha_r^\lambda,\lambda v_r)> \bar{L}_r(A_r\oplus \alpha_r,v_r)$$
and
$$\bar{L}_r((\lambda A_r+(1-\lambda)\Id)\oplus\alpha_r^\lambda,\lambda v_r)> \bar{L}_r(\bar{\Id},0),$$
which contradicts the fact that $\bar{L}_r$ is $convex^*$
(see Proposition \ref{prop14.3.*}). Therefore, $(M_r\oplus\beta_r,w_r)\in\bar{B}_{2R}$ for all $r\in(r_0,1)$ and we conclude the proof.  
\end{proof}

\begin{lema}\label{limit}
If $(M_0\oplus\beta_0,w_0)$ is the unique global minimum of $\bar{I}_1$, then $(M_r\oplus\beta_r,w_r)$ converges to $(M_0\oplus\beta_0,w_0)$.    
\end{lema}

\begin{proof}
Take $M\oplus\beta\in\s\R_{n+1,0}(\mathbb{R})$ and define
$$\s(M\oplus\beta)^{(r)}= \dfrac{\s\mbox{det}_{n+1}(\bar{\Id}+(1-r)(M\oplus\beta))^{-1/(n+s)}(\bar{\Id}+(1-r)(M\oplus\beta))-\bar{\Id}}{1-r}.$$
Note that $(\s(M\oplus\beta)^{(r)},w)$ belongs to 
$$\dfrac{\s\mathcal{E}_+\cap (\s\V\NN\RR)-\bar{\Id}\RR}{1-r}$$ for $r$ close to 1. We also have
\begin{align*}
\lim_{r\rightarrow 1^-} \s(M\oplus\beta)^{(r)}  &= \lim_{r\rightarrow 1^-} \left(\dfrac{\s\mbox{det}_{n+1}(\bar{\Id}+(1-r)(M\oplus\beta))^{-1/(n+s)}-1}{1-r}\bar{\Id}\right.\\
&\quad \left. + \ {}\s\mbox{det}_{n+1}(\bar{\Id}+(1-r)(M\oplus\beta))^{-1/(n+s)}M\oplus\beta\right)\\
& =  \left.\dfrac{\partial}{\partial t}\right|_{t=0}(\s\mbox{det}_{n+1}(\bar{\Id}+t(1-r)(M\oplus\beta))^{-1/(n+s)}\bar{\Id})+M\oplus\beta\\
& = \dfrac{-1}{n+s}\s\mbox{tr}(-M\oplus\beta)\bar{\Id}+M\oplus\beta\\
& = \ M\oplus\beta.
\end{align*}

By Lemma \ref{bounded},  the sequence $(M_r\oplus\beta_r,w_r)$ is bounded, then for every convergent subsequence $(M_{r_k}\oplus\beta_{r_k},w_{r_k})\rightarrow (M_0\oplus\beta_0,w_0)$ as $r_k\rightarrow 1^-$, and for every $(M\oplus\beta,w)\in\s\R_{n+1,0}(\mathbb{R})\RR$, we have    
$$\bar{I}_{r_k}(M_{r_k}\oplus\beta_{r_k},w_{r_k})\rightarrow \bar{I}_1(M_0\oplus\beta_0, w_0),$$
and 
$$\bar{I}_{r_k}(M_{r_k}\oplus \beta_{r_k}, w_{r_k})\leq \bar{I}_{r_k}(\s(M\oplus\beta)^{(r_k)},w)\rightarrow \bar{I}_1(M\oplus\beta, w).$$
Thus,  $(M_0\oplus\beta_0,w_0)$ is the (unique) minimum of $\bar{I}_1$,  and we conclude that $(M_r\oplus\beta_r,w_r)\rightarrow (M_0\oplus\beta_0,w_0)$ as required.
\end{proof}

By Lemma \ref{bounded}, the sequence \( (M_r \oplus \beta_r, w_r) \) is bounded. Therefore, for every convergent subsequence \( (M_{r_k} \oplus \beta_{r_k}, w_{r_k}) \to (M_0 \oplus \beta_0, w_0) \) as \( r_k \to 1^- \), and for every \( (M \oplus \beta, w) \in \s \mathbb{R}_{n+1,0}(\mathbb{R}) \mathbb{R} \), we have
\[
\bar{I}_{r_k}(M_{r_k} \oplus \beta_{r_k}, w_{r_k}) \to \bar{I}_1(M_0 \oplus \beta_0, w_0),
\]
and
\[
\bar{I}_{r_k}(M_{r_k} \oplus \beta_{r_k}, w_{r_k}) \leq \bar{I}_{r_k}(\s(M \oplus \beta)^{(r_k)}, w) \to \bar{I}_1(M \oplus \beta, w).
\]
Thus, \( (M_0 \oplus \beta_0, w_0) \) is the (unique) minimum of \( \bar{I}_1 \), and we deduce that \( (M_r \oplus \beta_r, w_r) \to (M_0 \oplus \beta_0, w_0) \) as desired.

\begin{proof}[Proof of Theorem \ref{teo9.3}]
By Lemma \ref{limit}, we have
\begin{align*}
\left.\dfrac{\partial (A_r\oplus\alpha_r,v_r)}{\partial r}\right|_{r=1}= \lim_{r\rightarrow 1^-}\dfrac{(A_r\oplus\alpha_r,v_r)-(\bar{\Id},0)}{r-1}=\lim_{r\rightarrow 1^-} (-M_r\oplus\beta_r,-w_r)=-(M_0\oplus \beta_0,w_0).
\end{align*}

Now, let  $\delta$ be any continuous function with compact support and, as in the proof of Theorem \ref{teo8.3}, consider the sets $\Lambda^c=\{x\in\mathbb{R}^n: |x|_2^2+h(x)^{2/s}>1 \},  \Lambda=\{x\in\mathbb{R}^n: |x|_2^2+h(x)^{1/s}= 1\}$. We have 
\begin{align*}
&\dfrac{1}{1-r}\int_{\mathbb{R}^n}\int_0^\infty \delta(x)(f')_r\left(\dfrac{y}{h(x)^{1/s}}\right)g_r\left(\dfrac{|\tilde{A}_r^{-1}(x-\tilde{v}_r)|_2^2+(\tilde{\alpha}_r^{-1}y)^2-1}{h(\tilde{A}_r^{-1}(x-\tilde{v}_r))^{2/s}}+1\right)\dfrac{y}{h(x)^{3/s}}\tilde{\alpha}_r^{s-1}dydx \\
& = \int_{\Lambda^c}\int_{-1}^\infty \delta(x)(f')_r(1+(1-r)t)g_r\left(\dfrac{|\tilde{A}_r^{-1}(x-\tilde{v}_r)|_2^2+(\tilde{\alpha}_r^{-1}(1+(1-r)t)h(x)^{1/s})^2-1}{h(\tilde{A}_r^{-1}(x-\tilde{v}_r))^{2/s}}+1\right)\\
& \qquad \times h(x)^{1/s}\dfrac{(1+(1-r)t)h(x)^{1/s}}{h(x)^{3/s}}\tilde{\alpha}_r^{s-1}dydx \\
& \quad + \int_{\Lambda}\int_{-1}^\infty \delta(x)(f')_r(1+(1-r)t)g_r\left(\dfrac{|\tilde{A}_r^{-1}(x-\tilde{v}_r)|_2^2+(\tilde{\alpha}_r^{-1}(1+(1-r)t)h(x)^{1/s})^2-1}{h(\tilde{A}_r^{-1}(x-\tilde{v}_r))^{2/s}}+1\right) \\
& \qquad \times h(x)^{1/s}\dfrac{(1+(1-r)t)h(x)^{1/s}}{h(x)^{3/s}}\tilde{\alpha}_r^{s-1}dydx\\
& =\int_{\Lambda^c}\int_{-1}^\infty \delta(x)f'(t)g\left(  \dfrac{|x|_2^2+h(x)^{2/s}-1 + (1-r)O(1)+ (1-r)t (2h(x)^{2/s}+o(1)) +o(1)}{2(1+(1-r)\beta)^2h(x-(1-r)(Mx+w)+o(1-r))^{2/s}(1-r)}\right)  \\
& \qquad \times \dfrac{(1+(1-r)t)}{h(x)^{1/s}}\tilde{\alpha}_r^{s-1}dydx  \\
& \quad + \int_{\Lambda}\int_{-1}^\infty \delta(x)f'(t)g\left(  \dfrac{-2\beta(1-|x|_2^2+o(1))-2\langle x,Mx+w +o(1)\rangle +t(2h(x)^{2/s}+o(1))+o(1)}{2(1+(1-r)\beta)^2h(x-(1-r)(Mx+w)+o(1-r))^{2/s}}\right) \\
& \qquad \times \dfrac{(1+(1-r)t)}{h(x)^{1/s}}\tilde{\alpha}_r^{s-1}dydx .
\end{align*}
Hence, by the Dominated Convergence Theorem, we obtain
\begin{align*}
\nonumber \dfrac{1}{1-r}\int_{\mathbb{R}^n}\int_0^\infty \delta(x)(f')_r\left(\dfrac{y}{h(x)^{1/s}}\right)g_r\left(\dfrac{|\tilde{A}_r^{-1}(x-v_r)|_2^2+(\tilde{\alpha}_r^{-1})y)^2-1}{h(\tilde{A}_r^{-1}(x-\tilde{v}_r))^{2/s}}+1\right)\dfrac{y}{h(x)^{3/s}}\tilde{\alpha}_r^{s-1}dydx \\
\longrightarrow \int_{\Lambda}\delta(x) \dfrac{1}{h(x)^{1/s}}F'\left(\dfrac{\langle x, M_0x+w_0\rangle}{h(x)^{2/s}}+\beta_0\right)dx
\end{align*}
as $r\rightarrow 1^-$.

Finally, since $(A_r\oplus\alpha_r, v_r)$ minimizes the functional $\bar{L}_r$,  by Lemma \ref{lema_auxialiar}, there exists $\lambda_r>0$ such that
\begin{align*}
\dfrac{1}{1-r}\int_{\mathbb{R}^n}\int_0^\infty &(f')_r\left(\dfrac{y}{h(x)^{1/s}}\right)g_r\left(\dfrac{|\tilde{A}_r^{-1}(x-v_r)|_2^2+(\tilde{\alpha}_r^{-1})y)^2-1}{h(\tilde{A}_r^{-1}(x-\tilde{v}_r))^{2/s}}+1\right)\dfrac{y}{h(x)^{3/s}}\tilde{\alpha}_r^{s-1}\\
& \times\left(- \nabla h(x)^{1/s}h(x)^{1/s}\otimes x\oplus h(x)^{1/s}h(x)^{1/s}\right)dydx = \lambda_r (\Id\oplus s,0)
\end{align*}
and
\begin{align*}
\dfrac{1}{1-r}\int_{\mathbb{R}^n}\int_0^\infty &(f')_r\left(\dfrac{y}{h(x)^{1/s}}\right)g_r\left(\dfrac{|\tilde{A}_r^{-1}(x-v_r)|_2^2+(\tilde{\alpha}_r^{-1})y)^2-1}{h(\tilde{A}_r^{-1}(x-\tilde{v}_r))^{2/s}}+1\right)\dfrac{y}{h(x)^{3/s}}\tilde{\alpha}_r^{s-1} \\
& \times \left(- \nabla h(x)^{1/s}h(x)^{1/s}\right)dydx = 0.
\end{align*}
By equations \eqref{eq5.30}, \eqref{eqI}, \eqref{eqII},  we conclude the desired result.
\end{proof}

\section*{Acknowledgments}
The author would like to thank Julián Haddad  for fruitful discussions and comments. This project was supported, in part, by the Austrian Science Fund (FWF) Grant-DOI: 10.55776/P34446.

\end{document}